\documentclass{birkjour}
\usepackage{amssymb}
\usepackage{amsthm}
\usepackage{mathrsfs}
\usepackage{mathtools}
\usepackage{esint}
 \newtheorem{thm}{Theorem}[section]
 
 \newtheorem{lem}[thm]{Lemma}
 \newtheorem{prop}[thm]{Proposition}
 \theoremstyle{definition}
 \newtheorem{defn}[thm]{Definition}
 \theoremstyle{remark}
 \newtheorem{rem}[thm]{Remark}
 
 \numberwithin{equation}{section}

\begin{document}

%-------------------------------------------------------------------------
% editorial commands: to be inserted by the editorial office
%
%\firstpage{1} \volume{228} \Copyrightyear{2004} \DOI{003-0001}
%\seriesextraline{This is the Concrete Title of this Book\br H.E. R and S.T.C. W, Eds.}
%\firstpage{1}

\title[\textit{Stoch. 2S. of Navier-Stokes Type Equations}]
 {Stochastic-periodic Homogenization of Non-stationary Incompressible Navier-Stokes Type Equations}

%----------Author 1
\author[\tiny{\textsc{Franck Tchinda}}]{\textsc{Franck Tchinda}}

\address{%
University of Maroua\\
Department of Mathematics and Computer Science\\
P.O. Box 814\\
Maroua, Cameroon}

\email{takougoumfranckarnold@gmail.com}

%----------Author 2
\author[\tiny{\textsc{Joel Fotso Tachago}}]{\textsc{Joel Fotso Tachago}}

\address{%
	University of Bamenda\\
	Higher Teachers Trainning College, Department of Mathematics\\
	P.O. Box 39\\
	Bambili, Cameroon}

\email{fotsotachago@yahoo.fr}

%\thanks{This work was completed with the support of our
%----------Author 3
\author[\tiny{\textsc{Joseph Dongho}}]{\textsc{Joseph Dongho}}
\address{%
	University of Maroua\\
	Department of Mathematics and Computer Science\\
	P.O. Box 814\\
	Maroua, Cameroon}
\email{joseph.dongho@fs.univ-maroua.cm}
\subjclass{35R60, 76D05, 35B27, 35B40.}
%35R60 PDEs with randomness, stochas. PDEs
%76D05 navier stokes equation for incompressible...
%35B27 homoge..of PDEs
%35B40 asymptotic behaviour...
\keywords{stochastic-periodic homogenization, stochastic two-scale convergence, dynamical system, non-stationary Navier-Stokes Equations, variational formulation, anisotropic heterogeneous media.}

\date{January 5, 2024}

\begin{abstract}
In this paper, we study the stochastic-periodic homogenization of non-stationary Navier-Stokes equations on anisotropic heterogeneous media. More precisely, we are interested in the stochastic-periodic homogenization of its variational formulation.  This problematic relies on the notion of dynamical system. It is shown by the stochastic two-scale convergence method that the resulting homogenized limit equation is of the same form of this variational formulation with suitable coefficients.\footnotetext{This work is supported by MINESUP/University of Maroua}
%A compactness theorem  is also proved in order to deal with time dependent homogenization problems.   
\end{abstract}

%%% ----------------------------------------------------------------------
\maketitle
%%% ----------------------------------------------------------------------
%\tableofcontents

%%%%%%%%%%%%%%%%%%%%%%%%%%%%%%%%%%%%%%%%%%%%%%%%%%%%%%%%%%%%%%%%%%%%%%%%%%%%%%%%%%%%%
%%%%%%%%%%%%%%%%%%%%    1ere PARTIE   %%%%%%%%%%%%%%%%%%%%%%%%%%%%%%%%%%%%%%%%%%%%%%%

\section{Introduction} \label{sect1} 

 It is well known that the flows of commonly encountered Newtonian fluids are modeled by the Navier-Stokes equations. These flows are sometimes laminar, sometimes turbulent. Unfortunately, in reality, the flows of fluids are almost always turbulent. Thus starting from two identical situations, the flow may evolve very differently. This explains its dual nature of being both deterministic and unpredictable (random). In this paper, we address the problem of the stochastic-periodic homogenization in heterogeneous media of non-stationary Navier-Stokes equations : 
\begin{equation}\label{tc11}
\begin{array}{rcll}
\rho \dfrac{\partial }{\partial t}\mathbf{u} + (\mathbf{u},\nabla)\mathbf{u} + \mathcal{E}(\mathbf{u}) + \nabla p & = & \mathbf{f} & \textup{in} \; [0,T]\times Q\times \Omega  \\
\textup{div}\,\mathbf{u} & = & 0 & \textup{in} \; [0,T]\times Q\times \Omega.
\end{array}
\end{equation}
 Let us specify the data in (\ref{tc11}). 
Let $\mathcal{T}$ be an $N$-dimensional dynamical system acting on the probability space $(\Omega, \mathscr{M}, \mu)$, $Q$ be a smooth bounded open subset of $\mathbb{R}^{N}_{x}$ (the $N$-dimensional numerical space $\mathbb{R}^{N}$ of variables $x=(x_{1},\cdots,x_{N})$), where $N$ is a given positive integer, and let $T>0$ be a real number.
 Here for all $(t,x) \in [0,T]\times Q$ and for almost all $\omega\in \Omega$, $\mathbf{u}(t, x, \omega) = (u_{1}, u_{2}, \cdots, u_{N})$ is the \textit{Eulerian velocity} of a fluid flow, $p(t, x, \omega)$ is the scalar value \textit{fluid pressures}, $\rho(x, \omega)$ is a scalar value function, $\mathbf{f}(t, x, \omega) = (f_{1}, f_{2}, \cdots, f_{N})$ is $\mathbb{R}^{N}$-value vector function  and $(\mathbf{u},\nabla)u_{k} = \sum_{i=1}^{N} u_{i}\frac{\partial}{\partial x_{i}}u_{k}, k=1,2,\cdots,N$. The \textit{elasticity} operator $\mathcal{E}$ is described later. 
\par The homogenized problems for fluid flows governed by Stokes and Navier-Stokes equations have been extensively studied by many authors (see, e.g. \cite{benou,bourgeat,signing,allair2,pak2,sor,temam}).\\
In \cite{benou}, BENSOUSSAN \& al. studied deterministic homogenization of stationary Navier-Stokes type equations in their linear form. Later NGUETSENG and SIGNING \cite{signing}, studied the periodic homogenization via $\Sigma$-convergence (see e.g. \cite{nguet2,wou2}) of the nonlinear version of this Navier-Stokes type equations. 
In \cite{allair2}, ALLAIRE study the periodic homogenization of stationary Stokes equations and Navier-Stokes equations in the perfored domains. In \cite{pak2}, PAK used the periodic two-scale convergence developed by NGUETSENG \cite{nguet1} and ALLAIRE \cite{allair1}, in order to homogenize the non-stationary Navier-Stokes equations in the form (\ref{tc11}). Note that its work takes into account only the deterministic setting in these equations. However, in 1994, BOURGEAT \cite{bourgeat} introduces the two-scale convergence in the mean, in order to homogenize the problems which depended only of a stochastic parameter. This notion will be developed later in \cite{sango}, by SANGO and WOUKENG which generalized on both periodic two-scale convergence and two-scale convergence in the mean methods, thus taking into account the deterministic and random aspects of natural phenomena. This combinaison is called \textit{stochastic two-scale convegence}. In \cite{tachago,tachago1}, TACHAGO and NNANG use this stochastic two-scale convergence method in the homogenization of Maxwell's equations with linear and nonlinear periodic conductivity. Later SANGO and WOUKENG developped in \cite{sango2} the stochastic $\Sigma$-convergence which generalized of both $\Sigma$-convergence and two-scale convergence in the mean methods. They used this method to homogenize stationary Navier-Stokes type equations.
\par Following PAK's work \cite{pak2}, we propose, in this paper, to use the stochastic two-scale converge method (see \cite{sango}) in order to homogenize the variational formulation of non-stationary Navier-Stokes equations (\ref{tc11}). The use of this method is motivated by the fact that in our case, the Navier-Stokes system takes into account the deterministic and random parameters at the same time. To do this, we start by constructing the \hbox{\LARGE$\epsilon$}-model for (\ref{tc11}) : 
\begin{equation}\label{tc12}
\begin{array}{rcl}
\rho_{m}\dfrac{\partial}{\partial t}\mathbf{u}_{m}^{\epsilon} + (\mathbf{u}_{m}^{\epsilon}, \nabla)\mathbf{u}_{m}^{\epsilon} + \mathcal{E}^{\varepsilon}_{m}(\mathbf{u}_{m}^{\epsilon}) + \nabla p_{m}^{\epsilon}& = &\mathbf{f}_{m},  \quad \textup{in} \;\, Q_{m}^{\epsilon}\times\Omega,  \\
\textup{div}\mathbf{u}_{m}^{\epsilon}(\cdot,\omega)& =& 0,  \;\, \textup{for\,a.e.}\; \omega\in \Omega 
\end{array}   
\end{equation}
with additional boundary conditions that we will be given in section \ref{sect4}. Here,
 $$\mathbf{u}^{\epsilon}_{m} \in L^{2}\left( [0,T] ; L^{2}(\Omega, H^{1}(Q^{\epsilon}_{m}))^{N} \right) \cap \mathcal{C}_{W}\left( [0,T] ; L^{2}(Q^{\epsilon}_{m}\times\Omega)^{N} \right),$$ 
 $$p^{\epsilon}_{m} \in L^{2}\left( [0,T] ; L^{2}(\Omega, H^{1}(Q^{\epsilon}_{m})) \right),$$
  $$\mathbf{f}_{m} \in L^{2}\left( [0,T] ; L^{2}(Q^{\epsilon}_{m}\times\Omega)^{N} \right),$$
  where $m=1, 2$ and $\epsilon>0$ is a small parameter. These functions spaces are described in section \ref{sect4}. The elasticity operator $\mathcal{E}^{\varepsilon}_{m}$, for  $m=1,2$,
is defined by 
$$\mathcal{E}^{\varepsilon}_{m}(\mathbf{u})_{i} = - \sum_{j=1}^{N} \sum_{k,l=1}^{N} \partial_{j}\, a^{m}_{ijkl}\left(x, \mathcal{T}\left(\frac{x}{\epsilon} \right)\omega, \frac{x}{\epsilon^{2}} \right) \varepsilon_{kl}(\mathbf{u}),\; 1 \leq i \leq N,$$
 where $a_{ijkl}$ is a positive definite symmetric \textit{elasticity} tensor and 
  the strain operator $\varepsilon$ is such that 
  $$\varepsilon_{kl}(\textbf{u}) \equiv \frac{1}{2} \left(\frac{\partial}{\partial x_{k}}u_{l} + \frac{\partial}{\partial x_{l}}u_{k} \right),\; 1\leq k,l \leq N.$$ 
 Thus the variational formulation of (\ref{tc12}) (see section \ref{sect4}) is given by 
\begin{equation}\label{c1tc9}
\begin{array}{l}
\sum_{m=1}^{2} \left[ \int_{Q^{\epsilon}_{m}\times\Omega}\rho_{m}\dfrac{\partial}{\partial t}\mathbf{u}^{\epsilon}_{m}\cdot\mathbf{v}_{m} + (\mathbf{u}^{\epsilon}_{m}, \nabla)\mathbf{u}^{\epsilon}_{m}\cdot\mathbf{v}_{m} \, dxd\mu + e^{\varepsilon}_{m}(\mathbf{u}_{m}, \mathbf{v}_{m}) \right] \\

\\
= \sum_{m=1}^{2} \int_{Q^{\epsilon}_{m}\times\Omega} \mathbf{f}_{m}\cdot\mathbf{v}_{m}\,dxd\mu
\end{array}
\end{equation}
and
\begin{equation}\label{c1tc10}
\sum_{m=1}^{2} \left[ \int_{Q^{\epsilon}_{m}\times\Omega} (\textup{div}\mathbf{u}^{\epsilon}_{m})\varphi_{m}\,dxd\mu \right] = 0
\end{equation} 
for all divergence free vector fields $\mathbf{v}_{m} \in \mathcal{C}^{1}\left( [0,T] ; L^{2}(\Omega, H^{1}(Q^{\epsilon}_{m}))^{N} \right)$, $m=1,2$ with $\mathbf{v}_{1} = \mathbf{v}_{2}$ on $\Gamma^{\epsilon}_{12}\times\Omega$ and all $\varphi_{m} \in \mathcal{C}^{1}\left( [0,T] ; L^{2}(\Omega, H^{1}(Q^{\epsilon}_{m}))^{N} \right)$, $m=1,2$ with 
$\varphi_{1} = \varphi_{2}$ on $\Gamma^{\epsilon}_{12}\times\Omega$. The operator $e^{\varepsilon}_{m}$ and the boundary $\Gamma^{\epsilon}_{12}$ are described in section \ref{sect4}. Moreover, taking into account the boundary conditions on $\mathbf{u}_{m}^{\epsilon}$ and $\mathbf{v}_{m}$, $m=1,2$, we show that the \hbox{\LARGE$\epsilon$}-weak formulation (\ref{c1tc9})-(\ref{c1tc10}) can be displayed as :   
\begin{equation}\label{c1rem11}
\begin{array}{l}
\int_{0}^{T}\int_{Q\times\Omega} \chi_{1}^{\epsilon}(x)  \rho_{1}(x,\omega)\mathbf{u}_{1}^{\epsilon}(t,x,\omega)\cdot\dfrac{\partial}{\partial t}\mathbf{v}^{\epsilon}_{1}(t,x,\omega)  +  \int_{0}^{T} e^{\varepsilon}_{1}(\mathbf{u}_{1}(t), \mathbf{v}_{1}^{\epsilon}(t))\,dt  \\
\qquad + \int_{0}^{T}\int_{Q\times\Omega} \chi_{1}^{\epsilon}(x) \{ \mathbf{u}^{\epsilon}_{1}(t,x,\omega)\otimes\mathbf{u}^{\epsilon}_{1}(t,x,\omega)\}\cdot\nabla \mathbf{v}_{1}^{\epsilon}(t,x,\omega)  dxd\mu dt \\
\qquad +\int_{0}^{T}\int_{Q\times\Omega} \chi_{2}^{\epsilon}(x)  \rho_{2}(x,\omega)\mathbf{u}_{2}^{\epsilon}(t,x,\omega)\cdot\dfrac{\partial}{\partial t}\mathbf{v}^{\epsilon}_{2}(t,x,\omega)  +  \int_{0}^{T} e^{\varepsilon}_{2}(\mathbf{u}_{2}(t), \mathbf{v}_{2}^{\epsilon}(t))\,dt  \\
\qquad + \int_{0}^{T}\int_{Q\times\Omega} \chi_{2}^{\epsilon}(x) \{ \mathbf{u}^{\epsilon}_{2}(t,x,\omega)\otimes\mathbf{u}^{\epsilon}_{2}(t,x,\omega)\}\cdot\nabla \mathbf{v}_{2}^{\epsilon}(t,x,\omega)  dxd\mu dt \\

\\
= \int_{Q_{1}^{\epsilon}\times\Omega} \rho_{1}(x,\omega)\mathbf{u}_{1}^{\epsilon}(0,x,\omega)\cdot\mathbf{v}_{1}^{\epsilon}(0,x,\omega) dxd\mu \\
\qquad + \int_{Q_{2}^{\epsilon}\times\Omega} \rho_{2}(x,\omega)\mathbf{u}_{2}^{\epsilon}(0,x,\omega)\cdot\mathbf{v}_{2}^{\epsilon}(0,x,\omega) dxd\mu \\
\qquad +\int_{0}^{T}\int_{Q\times\Omega} \chi_{1}^{\epsilon}(x) \mathbf{f}_{1}(t,x,\omega)\cdot\mathbf{v}_{1}^{\epsilon}(t,x,\omega) \,dxd\mu dt \\
\qquad + \int_{0}^{T}\int_{Q\times\Omega} \chi_{2}^{\epsilon}(x) \mathbf{f}_{2}(t,x,\omega)\cdot\mathbf{v}_{2}^{\epsilon}(t,x,\omega) \,dxd\mu dt,
\end{array}
\end{equation}
where $\chi_{m}^{\epsilon}(x)=\chi_{m}\left(\frac{x}{\epsilon^{2}}\right)$, with $\chi_{m}(y)$ the characteristic function of $Y_{m}$ for $m=1,2$ ($Y$ is given in complementary parts $Y_{1}$ and $Y_{2}$, see section \ref{sect4}). \\
 After this construction,  we are interested in the stochastic-periodic homogenization of the sequence of solutions of the variational problem (\ref{c1rem11}), i.e. the analysis of the asymptotic behaviour of the sequence of the solutions $\mathbf{u}^{\epsilon}_{m}, \,m=1,2$ when $\epsilon \rightarrow 0$. Thus, using the stochastic two-scale convergence method and referring to the notations in section \ref{sect4}, one of our fundamental results is the following : 
\par \textit{The sequence of solution $\mathbf{u}^{\epsilon}$ for the \hbox{\LARGE$\epsilon$}-weak formulation (\ref{c1rem11})  admits a subsequence still denoted  $\mathbf{u}^{\epsilon}$ such that, as $\epsilon\to 0$, we have : 
	\begin{itemize}
		\item[i)] $\mathbf{u}_{\epsilon} \rightarrow \mathbf{u}$ stoch. in $L^{2}(Q\times\Omega)^{N}$-weak 2s (componentwise),
		\item[ii)]  $\textup{div}\mathbf{u} \rightarrow \textup{div}\mathbf{u} + \textup{div}_{\omega}\mathbf{u}_{1} + \textup{div}_{y}\mathbf{u}_{2}$ stoch. in $L^{2}(Q\times\Omega)$-weak 2s,
		\item[iii)]  $\mathbf{e}(\mathbf{u}^{\epsilon}) \rightarrow \mathbf{e}(\mathbf{u}) + \mathbf{e}(\mathbf{u}_{1})_{\omega} + \mathbf{e}(\mathbf{u}_{2})_{y}$ stoch. in $L^{2}(Q\times\Omega)^{N}$-weak 2s (componentwise), with $\mathbf{e}(\mathbf{u})_{ij} = \varepsilon_{ij}(\mathbf{u}) = \frac{1}{2}\left( \dfrac{\partial \mathbf{u}_{i}}{\partial x_{j}} + \dfrac{\partial \mathbf{u}_{j}}{\partial x_{i}} \right)$, $\mathbf{e}(\mathbf{u})_{y} = \varepsilon_{ij}(\mathbf{u}) =  \frac{1}{2}\left( \dfrac{\partial \mathbf{u}_{i}}{\partial y_{j}} + \dfrac{\partial \mathbf{u}_{j}}{\partial y_{i}} \right)$, $\mathbf{e}(\mathbf{u})_{\omega} = \frac{1}{2}\left( \overline{D}^{\omega}_{j}\mathbf{u}_{i} + \overline{D}^{\omega}_{i}\mathbf{u}_{j} \right)$, $1\leq i,j\leq N$,
	\end{itemize}
	for some $\mathbf{u}(x,\omega) \in H^{1}(Q ; L^{2}_{nv}(\Omega))^{N}$, $\mathbf{u}_{1} \in L^{2}(Q ; H^{1}_{\#}(\Omega))^{N}$ and $\mathbf{u}_{2} \in L^{2}(Q\times\Omega ; H^{1}_{\#}(Y))^{N}$. Moreover we have $\textup{div} \mathbf{u}=0$, $\textup{div}_{\omega} \mathbf{u}_{1}=0$, $\textup{div}_{y} \mathbf{u}_{2}=0$ and $(\mathbf{u}, \mathbf{u}_{1}, \mathbf{u}_{2})$  are the solutions of the following homogenized weak formulation : }
\begin{equation}\label{c1rem13}
\begin{array}{l}
\int_{0}^{T}\int_{Q\times\Omega} \bigg\{  \rho(x,\omega) \dfrac{\partial}{\partial t}\mathbf{u}(t,x,\omega)\cdot \mathbf{v}(t,x,\omega) \\
\qquad\qquad \quad + \{ \mathbf{u}(t,x,\omega)\otimes\mathbf{u}(t,x,\omega)\}\cdot\nabla \mathbf{v}(t,x,\omega) \bigg\}dxd\mu dt  \\
\qquad +  \int_{0}^{T} e(\mathbf{u}(t) + \mathbf{u}_{1}(t) + \mathbf{u}_{2}(t), \mathbf{v}(t))\,dt \\

\\ 
= \int_{0}^{T}\int_{Q\times\Omega} \mathbf{f}(t,x,\omega)\cdot\mathbf{v}(t,x,\omega) \,dxd\mu dt
\end{array}
\end{equation}
\textit{with the effective coefficients given by}
\begin{equation}
\begin{array}{l}
\rho(x,\omega) = |Y_{1}|\rho_{1}(x,\omega) + |Y_{2}|\rho_{2}(x,\omega), \\

\\
\mathbf{f}(x,\omega) = |Y_{1}|\mathbf{f}_{1}(x,\omega) + |Y_{2}|\mathbf{f}_{2}(x,\omega)
\end{array}
\end{equation}
\textit{and the homogenized elasticity bilinear form is defined by,} 
\begin{equation}
\begin{array}{l}
e(\mathbf{u} + \mathbf{u}_{1} + \mathbf{u}_{2}, \mathbf{v}(t)) 
\equiv \sum_{i,j,k,l=1}^{N} \int_{Y}\int_{Q\times\Omega} a_{ijkl}(x,\omega,y) \bigg\{ \varepsilon_{kl}(\mathbf{u}) \\

\\
\qquad\qquad\qquad \qquad\qquad\qquad + \varepsilon_{kl}^{\omega}(\mathbf{u}_{1}) + \varepsilon_{kl}^{y}(\mathbf{u}_{2})  \bigg\}\varepsilon_{ij}(\mathbf{v})\,dxd\mu dy,
\end{array}
\end{equation}
\textit{where the corresponding effective elasticity tensor is }
\begin{equation}
a_{ijkl}(x,\omega,y) = \chi_{1}(y)a^{1}_{ijkl}(x,\omega,y) + \chi_{2}(y)a^{2}_{ijkl}(x,\omega,y).
\end{equation}
\par The paper is divided into sections each revolving around a specific aspect.
Section \ref{sect2} dwells on some preliminary results on stochastic two-scale convergence, and we prove a compactness theorem which is useful to deal with time-dependent problems. In section \ref{sect4}, we start by constructing the \hbox{\LARGE$\epsilon$}-model for the non-stationary Navier-Stokes equations (\ref{tc11}) and its weak formulation. Indeed, we present terminology to explain well-posedness and a-priori estimates. Lastly, the homogenized equations for the \hbox{\LARGE$\epsilon$}-weak formulation of non-stationary Navier-Stokes equations are investigated.

%%%%%%%%%%%%%%%%%%%%%%%%%%%%%%%%%%%%%%%%%%%%%%%%%%%%%%%%%%%%%%%%%%%%%%%%%%%%%%%%%%%%%
%%%%%%%%%%%%%%%%%%%%    2ere PARTIE   %%%%%%%%%%%%%%%%%%%%%%%%%%%%%%%%%%%%%%%%%%%%%%%

\section{Fundamentals of stochastic two-scale convergence}\label{sect2}

 Throughout this paper, all the vector spaces are assumed to be real vector spaces, and the scalar functions are assumed real valued.
We refer to  \cite{bourgeat}, \cite{sango}, \cite{sango2} for notations. \\
$Q$ is a open bounded set of $\mathbb{R}^{N}$, integer $N>1$ and
$|Q|$ denotes the volume of $Q$ with respect to the Lebesgue mesure.\\
$Y=[0,1]^{N}$ denote the unit cube of $\mathbb{R}^{N}$.\\
$E=(\epsilon_{n})_{n\in\mathbb{N}}$ is a \textit{fundamental sequence}. \\
$\mathcal{D}(Q)=\mathcal{C}^{\infty}_{0}(Q)$ is the vector space of smooth functions with compact support in $Q$. \\
$\mathcal{C}^{\infty}(Q)$ is the vector space of smooth functions on $Q$. \\
$\mathcal{K}(Q) $ is the vector space of continuous functions with compact support in $Q$. \\
$L^{p}(Q)$, integer $p \in [1,+\infty]$, is Lebesgue space of functions on $Q$.  \\
$W^{1,p}(Q) = \{ v \in L^{p}(Q) \, : \, \frac{\partial v}{\partial x_{i}} \in L^{p}(Q), \; 1 \leq i \leq N \}$, where derivatives are taken in the weak sense, is classical Sobolev's space of functions and $W^{1,2}(Q) \equiv H^{1}(Q)$. \\
$W^{1,p}_{0}(Q)$ the set of functions in $W^{1,p}(Q)$ with zero boundary condition and $W^{1,2}_{0}(Q) \equiv H^{1}_{0}(Q)$.\\
$\nabla$ or $D$ (resp. $\nabla_{y}$ or $D_{y}$) denote the (classical) gradient operator on $Q$ (resp. on $Y$). \\
 $\textup{div}$ (resp. $\textup{div}_{y}$) the (classical) divergence operator on $Q$ (resp. on $Y$).\\
If $F(\mathbb{R}^{N})$ is a given function space, we denote by $F_{per}(Y)$ the space of functions in $F_{loc}(\mathbb{R}^{N})$ that are $Y$-periodic, and by $F_{\#}(Y)$ those functions in $F_{per}(Y)$ with mean value zero. \\
$(\Omega, \mathscr{M}, \mu)$ : measure space with probability measure $\mu$. \\
$\{\mathcal{T}(y): \Omega\rightarrow \Omega \, , \, y \in \mathbb{R}^{N} \}$ is an $N$-dimensional dynamical system on $\Omega$  \\
$L^{p}(\Omega)$, integer $p \in [1,+\infty]$, is Lebesgue space of functions on $\Omega$.  \\
$L^{p}_{nv}(\Omega)$ is the set of all $\mathcal{T}$-invariant functions in $L^{p}(\Omega)$.\\
$D_{i}^{\omega} : L^{p}(\Omega)\rightarrow L^{p}(\Omega)$, $1 \leq i \leq N$, $\omega\in\Omega$ is the $i$-th stochastic derivative operator. \\
$D_{\omega} = (D_{1}^{\omega}, \cdots, D_{N}^{\omega})$ is the stochastic gradient operator on $\Omega$.\\
$W^{1,p}(\Omega) = \left\{ f \in L^{p}(\Omega) : D^{\omega}_{i}f \in L^{p}(\Omega) \, (1 \leq i \leq N) \right\}$,
where $D^{\omega}_{i}f$ is taken in the distributional sense on $\Omega$ and $W^{1,2}(\Omega) \equiv H^{1}(\Omega)$.\\
 $W_{\#}^{1,p}(\Omega)$ is the separated completion of $\mathcal{C}^{\infty}(\Omega)$ in $L^{p}(\Omega)$  with respect to the norm $\|f\|_{\#,p} = \left( \sum_{i=1}^{N} \|D^{\omega}_{i}f\|^{p}_{L^{p}(\Omega)} \right)^{\frac{1}{p}}$. \\
 $\textup{div}_{\omega}$ (the formal adjoint operator of $D_{\omega}$) is the stochastic divergence operator on $\Omega$. 
  
\subsection{Preliminary on dynamical systems and ergodic theory}

Let $(\Omega, \mathscr{M}, \mu)$ be a measure space with probability measure $\mu$. We define an $N$-dimensional dynamical system on $\Omega$ as a family $\{ \mathcal{T}(y) : y \in \mathbb{R}^{N} \}$	of invertible maps,
\begin{equation}
\mathcal{T}(y) : \Omega \longrightarrow \Omega
\end{equation}
such that for each $y \in \mathbb{R}^{N}$ both $\mathcal{T}(y)$ and $\mathcal{T}(y)^{-1}$ are measurable, and such that the following properties hold : 
\begin{itemize}
	\item[(i)]\textit{(group property)} $\mathcal{T}(0)=$ identity map on $\Omega$ and for all $y_{1},y_{2} \in \mathbb{R}^{N}$,
	\begin{equation}
	\mathcal{T}(y_{1}+y_{2})=\mathcal{T}(y_{1})\circ \mathcal{T}(y_{2}),
	\end{equation}
	"$\circ$" being the usual composition of mappings, and $0$ the origin in $\mathbb{R}^{N}$,
	\item[(ii)]\textit{(invariance)} for each $y \in \mathbb{R}^{N}$, the map $\mathcal{T}(y) : \Omega \rightarrow \Omega$ is measurable on $\Omega$ and $\mu$-measure preserving, i.e. $\mu(\mathcal{T}(y)F)=\mu(F)$, for all $F \in \mathscr{M}$,
	\item[(iii)]\textit{(measurability)} for each $F \in \mathscr{M}$, the set $\{ (y,\omega) \in \mathbb{R}^{N}\times \Omega : \mathcal{T}(y)\omega \in F \}$ is measurable with respect to the product $\sigma$-algebra $\mathscr{L}\otimes \mathscr{M}$, where $\mathscr{L}$ is the $\sigma$-algebra of Lebesgue measurable sets. 
\end{itemize}
If $\Omega$ is a compact topological space, by a continuous $N$-dynamical system on $\Omega$, we mean any $N$-dynamical system on $\Omega$ above the following condition :
\begin{itemize}
	\item[(iv)] \textit{(continuity)} the mapping $(y, \omega) \mapsto \mathcal{T}(y)\omega$ is continuous from $\mathbb{R}^{N}\times \Omega$ to $\Omega$.
\end{itemize} 

Given $1 \leq p \leq +\infty$, a dynamical system $\{ \mathcal{T}(y) : \Omega \rightarrow \Omega, \, y \in \mathbb{R}^{N} \}$ induces a strongly continuous $N$-parameter group of isometries $U(y) : L^{p}(\Omega)\rightarrow L^{p}(\Omega)$, $y \in \mathbb{R}^{N}$, defined by : $(U(y)f)(\omega) = f(\mathcal{T}(y)\omega)$, $f \in L^{p}(\Omega)$ ; for more details, see, e.g., \cite{jikov}, \cite{pankov}. In particular the strong continuity is expressed by $\|U(y)f - f\|_{L^{p}(\Omega)} \rightarrow 0$ as $|y|\to 0$. \\
The $i$-th stochastic derivative $D_{i}^{\omega} : L^{p}(\Omega)\rightarrow L^{p}(\Omega)$, $1 \leq i \leq N$ is an unbounded linear mapping of domain 
\begin{equation*}
\mathbf{D}_{i} = \left\{ f \in L^{p}(\Omega) : \lim_{\iota \to 0} \dfrac{U(\iota e_{i})f - f}{\iota} \, \textup{exists\, in}\, L^{p}(\Omega) \right\},
\end{equation*} 
$(e_{i})$ being the canonical basis of $\mathbb{R}^{N}$, such that for all $f \in \mathbf{D}_{i}$, 
\begin{equation*}
D_{i}^{\omega}f(\omega) = \lim_{\iota \to 0} \dfrac{f(\mathcal{T}(\iota e_{i})\omega) - f(\omega)}{\iota},
\end{equation*}
for almost all $\omega \in \Omega$. \\
Higher order stochastic derivatives can be recalled analogously by setting $D_{\omega}^{\alpha} = D_{1}^{\alpha_{1},\omega}\circ \cdots \circ D_{N}^{\alpha_{N},\omega}$ for every multi-index $\alpha = (\alpha_{1}, \cdots, \alpha_{N}) \in \mathbb{N}^{N}$, where $D_{i}^{\alpha_{i},\omega} = D_{i}^{\omega} \circ \cdots \circ D_{i}^{\omega}$ ($\alpha_{i}$-times). \\
This being so, put $\mathcal{D}_{p}(\Omega) = \cap_{i=1}^{N}\mathbf{D}_{i}$, define 
\begin{equation*}
\mathcal{D}_{p}^{\infty}(\Omega) = \left\{ f \in L^{p}(\Omega) : D_{\omega}^{\alpha}f \in \mathcal{D}_{p}(\Omega) \, \textup{for \, all}\, \alpha \in \mathbb{N}^{N}  \right\},
\end{equation*}
and thanks to \cite{andre}, it can be shown that $\mathcal{D}_{\infty}^{\infty}(\Omega) \equiv \mathcal{C}^{\infty}(\Omega)$ is dense in $L^{p}(\Omega)$ for $1\leq p < +\infty$. \\
Endowed with a suitable locally convex topology, induced by the family of seminorms $	N_{n}(f) = \sup_{|\alpha|\leq n} \sup_{\omega\in\Omega} |D_{\infty}^{\alpha}f(\omega)|$, where $|\alpha|= \alpha_{1}+ \cdots + \alpha_{N}$, $\mathcal{C}^{\infty}(\Omega)$ is a Fr\'{e}chet space. Any continuous linear form on $\mathcal{C}^{\infty}(\Omega)$ is referred to as a stochastic distribution ; and the space $(\mathcal{C}^{\infty}(\Omega))'$ of all stochastic distributions is endowed with the strong dual topology. \\
The stochastic weak derivative of multi-index $\alpha$ ($\alpha \in \mathbb{N}^{N}$) of an element $S \in (\mathcal{C}^{\infty}(\Omega))'$ is the stochastic distribution $D^{\alpha}_{\omega}S$ given by $\langle D^{\alpha}_{\omega}S, \varphi \rangle = (-1)^{|\alpha|} \langle S, D^{\alpha}_{\omega}\varphi \rangle$, $\varphi \in \mathcal{C}^{\infty}(\Omega)$. Moreover, by the above density, $L^{p}(\Omega)$, $1\leq p < +\infty$ is a subspace of $(\mathcal{C}^{\infty}(\Omega))'$ with continuous embedding. So, the stochastic derivative of $f \in L^{p}(\Omega)$ exists with $\langle D^{\alpha}_{\omega}f, \varphi \rangle = (-1)^{|\alpha|} \int_{\Omega} f D^{\alpha}_{\omega}\varphi d\mu$ for all  $\varphi \in \mathcal{C}^{\infty}(\Omega)$. Therefore, we define the Sobolev on $\Omega$ as follows (see \cite{sango}) : 
\begin{equation*}
W^{1,p}(\Omega) = \left\{ f \in L^{p}(\Omega) : D^{\omega}_{i}f \in L^{p}(\Omega) \, (1 \leq i \leq N) \right\},
\end{equation*}
where $D^{\omega}_{i}f$ is taken in the distribution sense on $\Omega$. Equiped with the norm 
\begin{equation*}
\|f\|_{W^{1,p}(\Omega)} = \left( \|f\|^{p}_{L^{p}(\Omega)} + \sum_{i=1}^{N}\|D_{i}^{\omega}f\|^{p}_{L^{p}(\Omega)} \right)^{\frac{1}{p}},  \quad f \in W^{1,p}(\Omega),
\end{equation*} 
$W^{1,p}(\Omega)$ is a Banach space. Instead of $W^{1,p}(\Omega)$ we will be concerned with one of its seminormed subspace. At this juncture, let us recall that : 
\par  \textit{ (v) (property)} Given $f \in \mathcal{D}_{1}^{\infty}(\Omega)$ and for a.e. $\omega \in \Omega$, the function $y \rightarrow f(\mathcal{T}(y)\omega)$ lies in $\mathcal{C}^{\infty}(\mathbb{R}^{N})$ and $D^{\alpha}_{y}f(\mathcal{T}(y)\omega) = D^{\alpha}_{\omega}f(\mathcal{T}(y)\omega)$, for any $\alpha \in \mathbb{N}^{N}$. 
\par \textit{(vi) (definition)} By $f \in L^{p}(\Omega)$, $1 \leq p \leq \infty$, is $T$-invariant is meant for any $y \in \mathbb{R}^{N}$, $f(\mathcal{T}(y)\omega) = f(\omega)$ for  a.e. $\omega \in \Omega$. 
Denoting by  $L^{p}_{nv}(\Omega)$, the set of all $\mathcal{T}$-invariant functions in $L^{p}(\Omega)$, the dynamical system $\mathcal{T}$ is termed to be ergodic if $L^{p}_{nv}(\Omega)$ is the set of constant functions. For a.e. $\omega \in \Omega$, the function $y \rightarrow f(\mathcal{T}(y)\omega)$, $y \in \mathbb{R}^{N}$, is called a realization of $f$ and the mapping $(y,\omega) \rightarrow f(\mathcal{T}(y)\omega)$ is called a stationary process. The process is said to be stationary ergodic if the dynamical system $\mathcal{T}$ is ergodic.  
\par  \textit{(vii) (property)} Given $f \in L^{1}(\Omega)$, one has $f \in L^{1}_{nv}(\Omega)$ if and only if $D^{\omega}_{i}f = 0$ for each $1 \leq i \leq N$. \\
According to \textit{(vii)}, we consider $\mathcal{C}^{\infty}(\Omega)$ provided with the seminorm 
\begin{equation*}
\|f\|_{\#,p} = \left( \sum_{i=1}^{N} \|D^{\omega}_{i}f\|^{p}_{L^{p}(\Omega)} \right)^{\frac{1}{p}}, \quad f \in \mathcal{C}^{\infty}(\Omega), \, 1 < p < \infty. 
\end{equation*} 
So topologized, $\mathcal{C}^{\infty}(\Omega)$ is in general nonseparated an noncomplete. We denote by $W_{\#}^{1,p}(\Omega)$ the separated completion of $\mathcal{C}^{\infty}(\Omega)$ and by $I_{p}$ the canonical mapping of $\mathcal{C}^{\infty}(\Omega)$ into its separated completion (see, e.g., \cite[Chapter II]{barki}). Furthermore, as pointed out in \cite{sango} (see also \cite[Chapter II]{barki}), the distributional stochastic derivative $D^{\omega}_{i}$, $1 \leq i \leq N$, viewed as a mapping of $\mathcal{C}^{\infty}(\Omega)$ into $L^{p}(\Omega)$ extends to a unique continuous linear mapping, still denoted by $\overline{D}^{\omega}_{i}$, of $W_{\#}^{1,p}(\Omega)$ into $L^{p}(\Omega)$ such that $\overline{D}^{\omega}_{i}\circ I_{p}(v) = D^{\omega}_{i}v$ for $v \in \mathcal{C}^{\infty}(\Omega)$ and 
\begin{equation*}
\|u\|_{W_{\#}^{1,p}(\Omega)} = \left( \sum_{i=1}^{N} \|\overline{D}^{\omega}_{i}u\|^{p}_{L^{p}(\Omega)} \right)^{\frac{1}{p}}, \quad \textup{for} \, u \in W_{\#}^{1,p}(\Omega). 
\end{equation*} 
Moreover, $W_{\#}^{1,p}(\Omega)$ is a reflexive Banach space by the fact that the stochastic gradient $\overline{D}_{\omega} = (\overline{D}_{1}^{\omega}, \cdots, \overline{D}_{N}^{\omega})$ sends isometrically $W_{\#}^{1,p}(\Omega)$ into $L^{p}(\Omega)^{N}$. By duality, the operator $\textup{div}_{\omega} : L^{p'}(\Omega)^{N} \rightarrow (W_{\#}^{1,p}(\Omega))'$, $(p' = \frac{p}{p-1})$ defined by 
\begin{equation*}
\langle \textup{div}_{\omega}u, v \rangle = - \langle u, \overline{D}_{\omega}v \rangle =  - \sum_{i=1}^{N} \int_{\Omega} u_{i} \overline{D}_{i}^{\omega}v \,d\mu
\end{equation*}
$(u = (u_{i}) \in L^{p'}(\Omega)^{N}, v \in W_{\#}^{1,p}(\Omega))$ naturally extends the stochastic divergence operator in $\mathcal{C}^{\infty}(\Omega)$, and we have the following fundamental lemma. 
\begin{lem}
	Let $\mathbf{v} = (v_{i}) \in L^{p}(\Omega)^{N}$ such that $\sum_{i=1}^{N} \int_{\Omega} \mathbf{v}\cdot\mathbf{g}\,d\mu = 0$ for all $\mathbf{g} \in \mathcal{V}^{\omega}_{\textup{div}} = \{ \mathbf{f} = (f_{i})\in \mathcal{C}^{\infty}(\Omega)^{N} : \textup{div}_{\omega}\mathbf{f}=0 \}$. Then $\textup{curl}_{\omega}\mathbf{v}=0$ and there exists $u \in W_{\#}^{1,p}(\Omega)$ such that $\mathbf{v} = \overline{D}_{\omega}u$.
\end{lem}
\begin{proof}
	See \cite[lemma 2.3]{bourgeat} (or \cite[Proposition 1]{sango}).
\end{proof}

In the sequel, we will put $W^{1,2}(\Omega) = H^{1}(\Omega)$ and $W_{\#}^{1,2}(\Omega)= H^{1}_{\#}(\Omega)$.

\subsection{stochastic two-scale convergence}

Throughout the paper the letter $E$ will denote any ordinary sequence $E=(\epsilon_{n})$ (integers $n\geq 0$) with $0 \leq \epsilon_{n} \leq 1$ and $\epsilon_{n} \rightarrow 0$ as $n\rightarrow \infty$. Such a sequence will be termed a \textit{fundamental sequence}. In this section, we recall the concept of stochastic two-scale convergence which is the generalization of both two-scale convergence in the mean (of Bourgeat and al. \cite{bourgeat}) and two-scale convergence (of Nguetseng \cite{nguet1}). In all that follows, $Q$ is an open bounded subset of $\mathbb{R}^{N}$, $\{ \mathcal{T}(y), \, y \in \mathbb{R}^{N} \}$ is a fixed $N$-dimensional dynamical system on $\Omega$, with respect to a fixed probability measure $\mu$ on $\Omega$ invariant for $\mathcal{T}$. Let $Y = (0,1)^{N}$ and let $F(\mathbb{R}^{N})$ be a given function space. We also assume that $L^{p}(\Omega)$ is separable. We denote by $F_{per}(Y)$ the space of functions in $F_{loc}(\mathbb{R}^{N})$ that are $Y$-periodic, and by $F_{\#}(Y)$ those functions in $F_{per}(Y)$ with mean value zero. As special cases, $\mathcal{D}_{per}(Y)$ denotes the space $\mathcal{C}^{\infty}_{per}(Y)$ while $\mathcal{D}_{\#}(Y)$ stands for the space of those functions in $\mathcal{D}_{per}(Y)$ with mean value zero. $\mathcal{D}'_{per}(Y)$ stands for the topological dual of $\mathcal{D}_{per}(Y)$ which can be identified with the space of periodic distributions in $\mathcal{D}'(\mathbb{R}^{N})$. We say that an element $f \in L^{p}(Q\times\Omega\times Y)$ is \textit{admissible} if the function $f_{\mathcal{T}} : (x, \omega, y) \rightarrow f(x, \mathcal{T}(x)\omega, y)$, $(x, \omega, y) \in Q\times\Omega\times Y$, defines an element of $L^{p}(Q\times\Omega\times Y)$. As example, every element of $\mathcal{C}^{\infty}_{0}(Q)\otimes\mathcal{C}^{\infty}(\Omega)\otimes\mathcal{C}^{\infty}_{per}(Y) \subseteq L^{p}(Q)\otimes L^{p}(\Omega) \otimes L^{p}_{per}(Y)$ is \textit{admissible}.

\begin{defn}
	Let $(u_{\epsilon})_{\epsilon>0}$ be a bounded sequence in $L^{p}(Q\times\Omega)$, $1\leq p < \infty$. The sequence $(u_{\epsilon})_{\epsilon>0}$ is said to weakly stochastically two-scale converge in $L^{p}(Q\times\Omega)$ to some $u_{0} \in L^{p}(Q\times\Omega; L^{p}_{per}(Y))$ if as $\epsilon\to 0$, we have 
	\begin{equation}
	\int_{Q\times\Omega} u_{\epsilon}(x,\omega)f\left(x, \mathcal{T}\left(\frac{x}{\epsilon}\right)\omega, \frac{x}{\epsilon^{2}} \right) dxd\mu \rightarrow \iint_{Q\times\Omega\times Y} u_{0}(x,\omega,y)f(x,\omega,y) dx d\mu dy
	\end{equation}
	for every $f \in \mathcal{C}^{\infty}_{0}(Q)\otimes\mathcal{C}^{\infty}(\Omega)\otimes\mathcal{C}^{\infty}_{per}(Y)$. \\
	We can express this by writing $u_{\epsilon} \rightarrow u_{0}$ stoch. in $L^{p}(Q\times\Omega)$-weak 2s. 
\end{defn}

We recall that $\mathcal{C}^{\infty}_{0}(Q)\otimes\mathcal{C}^{\infty}(\Omega)\otimes\mathcal{C}^{\infty}_{per}(Y)$ is the space of functions of the form, 
\begin{equation}
f(x, \omega, y) = \sum_{finite} \varphi_{i}(x)\psi_{i}(\omega)g_{i}(y), \quad (x,\omega, y) \in Q\times \Omega\times \mathbb{R}^{N} ,
\end{equation}
with $\varphi_{i} \in \mathcal{C}^{\infty}_{0}(Q)$, $\psi_{i} \in \mathcal{C}^{\infty}(\Omega)$ and $g_{i} \in \mathcal{C}^{\infty}_{per}(Y)$. Such functions are dense in $\mathcal{C}^{\infty}_{0}(Q) \otimes L^{p'}(\Omega)\otimes\mathcal{C}^{\infty}_{per}(Y)$ ($p'= \frac{p}{p-1}$) for $1< p < \infty$, since $\mathcal{C}^{\infty}(\Omega)$ is dense in $L^{p'}(\Omega)$ and hence $\mathcal{K}(Q; L^{p'}(\Omega))\otimes\mathcal{C}^{\infty}_{per}(Y)$, where $\mathcal{K}(Q; L^{p'}(\Omega))$ being the space of continuous functions of $Q$ into $L^{p'}(\Omega)$ with compact support containing in $Q$, (see e.g., \cite[Chap III, Proposition 5]{barki2} for denseness result). As $\mathcal{K}(Q; L^{p'}(\Omega))$ is dense in $L^{p'}(Q; L^{p'}(\Omega)) = L^{p'}(Q\times\Omega)$ and $L^{p'}(Q\times\Omega)\otimes\mathcal{C}^{\infty}_{per}(Y)$ is dense in $L^{p'}(Q\times\Omega ; \mathcal{C}_{per}(Y))$, the uniqueness of the stochastic two-scale limit is ensured. 
\par Now, we give some properties about the stochastic two-scale convergence whose proof can be found in \cite{sango} or \cite{sango2}.
\begin{prop}\cite{sango} \\
	Let $(u_{\epsilon})_{\epsilon>0}$ be a sequence in $L^{p}(Q\times\Omega)$. If $u_{\epsilon} \rightarrow u_{0}$ stoch. in $L^{p}(Q\times\Omega)$-weak 2s, then $(u_{\epsilon})_{\epsilon>0}$ stochastically two-scale converges in the mean ( see \cite{bourgeat}) towards $v_{0}(x,\omega) = \int_{Y} u_{0}(x,\omega,y) dy$ and 
	\begin{equation}
	\int_{\Omega} u_{\epsilon}(\cdot, \omega)\psi(\omega)d\mu \rightarrow \iint_{\Omega\times Y} u_{0}(\cdot, \omega, y)d\mu dy \; \textup{in} \, L^{1}(Q)-weak \; \forall \psi \in L^{p'}_{nv}(\Omega).
	\end{equation}
\end{prop}
\begin{prop}\cite{sango}\label{rem3} \\
	Let $f \in \mathcal{K}(Q; \mathcal{C}_{per}(Y; \mathcal{C}^{\infty}(\Omega)))$. Then, as $\epsilon\to 0$, 
	\begin{equation}
	\int_{Q\times\Omega} \left| f\left(x, \mathcal{T}\left(\frac{x}{\epsilon}\right)\omega, \frac{x}{\epsilon^{2}} \right)\right|^{p} dxd\mu \rightarrow \iint_{Q\times\Omega\times Y} \left|f(x,\omega,y)\right|^{p} dx d\mu dy,
	\end{equation}
	for $1 \leq p < \infty$.
\end{prop}
\begin{prop}\cite{sango}\label{rem1} \\
	Any bounded sequence $(u_{\epsilon})_{\epsilon\in E}$  in $L^{p}(Q\times\Omega)$ admits a subsequence which is weakly stochastically two-scale convergent in $L^{p}(Q\times\Omega)$.
\end{prop}
\begin{prop}\cite{sango}\label{rem12} \\
	Let $1 < p < \infty$. Let $X$ be a norm closed convex subset of $W^{1,p}(Q)$, $Q$ being an open bounded subset of $\mathbb{R}^{N}$. Assume that $(u_{\epsilon})_{\epsilon\in E}$ is a sequence in $L^{p}(Q\times\Omega)$ such that : 
	\begin{itemize}
		\item[i)] $u_{\epsilon}(\cdot, \omega) \in X$ for all $\epsilon \in E$ and for $\mu$-a.e. $\omega \in \Omega$ ;
		\item[ii)] $(u_{\epsilon})_{\epsilon\in E}$ is bounded in $L^{p}(\Omega; W^{1,p}(Q))$.
	\end{itemize}
	Then there exist $u_{0} \in W^{1,p}(Q; L_{nv}^{p}(\Omega))$, $u_{1} \in L^{p}(Q; W^{1,p}_{\#}(\Omega))$, $u_{2} \in L^{p}(Q\times \Omega ; W^{1,p}_{\#}(Y))$ and a subsequence $E'$ from $E$ such that 
	\begin{itemize}
		\item[iii)] $u_{0}(\cdot, \omega) \in X$ for $\mu$-a.e. $\omega \in \Omega$ and, as $E' \ni \epsilon \rightarrow 0$, 
		\item[iv)] $u_{\epsilon} \rightarrow u_{0}$ stoch. in $L^{p}(Q\times \Omega)$-weak 2s;
		\item[v)] $Du_{\epsilon} \rightarrow Du_{0} + \overline{D}_{\omega}u_{1} + D_{y}u_{2}$ stoch. in $L^{p}(Q\times \Omega)^{N}$-weak 2s.  
	\end{itemize}
\end{prop}
From this proposition \ref{rem12}, we deduce the other following proposition.
\begin{prop}\label{rem5} 
	Let $1 < p < \infty$, $Q$ being an open bounded subset of $\mathbb{R}^{N}$. Assume that $(v_{\epsilon})_{\epsilon\in E}$ is a sequence in $L^{p}(Q\times\Omega)^{N}$ such that : 
	\begin{itemize}
		\item[i)] $(v_{\epsilon})$ is bounded in $L^{p}(Q\times\Omega)^{N}$;
		\item[ii)] $(\textup{div}v_{\epsilon})$ is bounded in $L^{p}(Q\times\Omega)$.
	\end{itemize}
	Then there exist $v_{0} \in W^{1,p}(Q; L_{nv}^{p}(\Omega))^{N}$, $v_{1} \in L^{p}(Q; W^{1,p}_{\#}(\Omega))^{N}$, $v_{2} \in L^{p}(Q\times \Omega ; W^{1,p}_{\#}(Y))^{N}$ and a subsequence $E'$ from $E$ such that, as $E' \ni \epsilon \rightarrow 0$,  
	\begin{itemize}
		\item[iii)] $v_{\epsilon} \rightarrow v_{0}$ stoch. in $L^{p}(Q\times \Omega)^{N}$-weak 2s; 
		\item[iv)] $\textup{div}v_{\epsilon} \rightarrow \textup{div}v_{0} + \textup{div}_{\omega}v_{1} + \textup{div}_{y}v_{2}$ stoch. in $L^{p}(Q\times \Omega)$-weak 2s. 
	\end{itemize}
\end{prop}

In order to deal with the homogenized property for parabolic equations, we prove the following compactness theorem with one parameter, which generalizes the compactness theorem for elliptic problems described in proposition \ref{rem1}.
\begin{thm}\label{rem2}
	Let $0 < T \leq \infty$. Suppose that $(u_{\epsilon})_{\epsilon\in E}$ is a bounded sequence of functions in $L^{2}([0,T); L^{p}(Q\times\Omega))$. Then there are a subsequence $E'$ from $E$ and $u_{0}(t,x,\omega,y) \in L^{2}([0,T); L^{p}(Q\times\Omega; L^{p}_{per}(Y)))$ such that for any $f(t,x,\omega,y) \in L^{2}([0,T); \mathcal{C}^{\infty}_{0}(Q)\otimes\mathcal{C}^{\infty}(\Omega)\otimes\mathcal{C}^{\infty}_{per}(Y))$, as $\epsilon\to 0$ we have 
	\begin{equation}
	\begin{array}{ll}
	\int_{0}^{T} \int_{Q\times\Omega} u_{\epsilon}(t,x,\omega)f\left(t,x, \mathcal{T}\left(\frac{x}{\epsilon}\right)\omega, \frac{x}{\epsilon^{2}} \right) dxd\mu dt & \\
	&  \\
	\longrightarrow  \int_{0}^{T}\iint_{Q\times\Omega\times Y} u_{0}(t,x,\omega,y)f(t,x,\omega,y) dx d\mu dy dt &
	\end{array}
	\end{equation}	
\end{thm}
The proof of the above theorem relies on the following proposition whose proof can be found in \cite{gabri}. 
\begin{prop}\cite[Proposition 3.2]{gabri}\label{rem4} \\
	Let $F$ be a subspace (not necessarily closed) of a reflexive Banach space $G$ and let $f_{n} : F \rightarrow \mathbb{C}$ be a sequence of linear functionals (not necessarily continuous).  Assume there exists a constant $C > 0$ such that 
	\begin{equation*}
	\limsup_{n} |f_{n}(x)| \leq C \, \|x\| \quad \textup{for \,\, all} \,\, x \in F,
	\end{equation*}
	where $\|\cdot\|$ denotes the norm in $G$. Then there exists a subsequence $(f_{n_{k}})_{k}$ of $(f_{n})$ and a functional $f \in G'$ such that 
	\begin{equation*}
	\lim_{k} f_{n_{k}}(x) = f(x) \quad \textup{for \,\, all} \,\, x \in F,
	\end{equation*} 
\end{prop} 
\begin{proof}(of Theorem \ref{rem2})  \\
	We have a given bounded sequence $(u_{\epsilon})_{\epsilon\in E}$ in $L^{2}([0,T); L^{p}(Q\times\Omega))$. There is a positive constant $C$ such that $\|u_{\epsilon}\|_{L^{2}([0,T); L^{p}(Q\times\Omega))} \leq C$.
	\\  Let us define $\mathcal{D} =  \mathcal{C}^{\infty}_{0}(Q)\otimes\mathcal{C}^{\infty}(\Omega)\otimes
	\mathcal{C}^{\infty}_{per}(Y)$ and the mapping
	\begin{equation*}
	L_{\epsilon}(f) = \int_{0}^{T} \int_{Q\times\Omega} u_{\epsilon}f^{\epsilon} dxd\mu dt, \quad f \in L^{2}([0,T); \mathcal{D})
	\end{equation*}
	where $f^{\epsilon}(t,x,\omega) = f\left(t,x, \mathcal{T}\left(\frac{x}{\epsilon}\right)\omega, \frac{x}{\epsilon^{2}} \right)$ for $(t,x,\omega) \in [0,T)\times Q\times\Omega$. 
	\par	Let now $G = L^{2}([0,T); L^{p}(Q\times\Omega; L^{p}_{per}(Y)))$ and $F = L^{2}([0,T); \mathcal{D})$. Then $G$ is a reflexive Banach space and $F$ is a subspace of $G$. 
	\par Moreover, using H\"{o}lder's inequality and the proposition \ref{rem3} we have
	\begin{equation*}
	\begin{array}{rcl}
	|L_{\epsilon}(f)| & \leq & \int_{0}^{T} \|u_{\epsilon}(t)\|_{L^{p}(Q\times\Omega)} \lVert f\left(t,x, \mathcal{T}\left(\frac{x}{\epsilon}\right)\omega, \frac{x}{\epsilon^{2}} \right)\rVert_{L^{p'}(Q\times\Omega)} dt \\
	& \leq & \left( \int_{0}^{T} \|u_{\epsilon}(t)\|^{2}_{L^{p}(Q\times\Omega)} dt \right)^{1/2} \times \left(\int_{0}^{T} \lVert f\left(t,x, \mathcal{T}\left(\frac{x}{\epsilon}\right)\omega, \frac{x}{\epsilon^{2}} \right)\rVert^{2}_{L^{p'}(Q\times\Omega)} dt \right)^{1/2} \\
	& \leq & C \times \left(\int_{0}^{T} \lVert f\left(t,x, \mathcal{T}\left(\frac{x}{\epsilon}\right)\omega, \frac{x}{\epsilon^{2}} \right)\rVert^{2}_{L^{p'}(Q\times\Omega)} dt \right)^{1/2}. 
	\end{array}
	\end{equation*} 
	Thus,
	\begin{equation*}
	\begin{array}{rcl}
	\limsup_{\epsilon\to 0} |L_{\epsilon}(f)| & \leq & C \times \left(\int_{0}^{T} \lVert f\left(t,x, \omega, y \right)\rVert^{2}_{L^{p'}(Q\times\Omega\times Y)} \right)^{1/2}  \\
	& = & C \times \|f\|_{L^{2}([0,T) ; L^{p'}(Q\times\Omega\times Y))}.
	\end{array}
	\end{equation*} 
	We deduce from the proposition \ref{rem4} the existence of a subsequence $E'$ of $E$ and of a unique $u_{0} \in G' = L^{2}([0,T); L^{p}(Q\times\Omega; L^{p}_{per}(Y)))$ such that,  as $\epsilon\to 0$, 
	\begin{equation}
	\begin{array}{ll}
	\int_{0}^{T} \int_{Q\times\Omega} u_{\epsilon}(t,x,\omega)f\left(t,x, \mathcal{T}\left(\frac{x}{\epsilon}\right)\omega, \frac{x}{\epsilon^{2}} \right) dxd\mu dt & \\
	 &  \\
	 \longrightarrow  \int_{0}^{T}\iint_{Q\times\Omega\times Y} u_{0}(t,x,\omega,y)f(t,x,\omega,y) dx d\mu dy dt &
	\end{array}
	\end{equation}
	for all $f(t,x,\omega,y) \in F$, whence the theorem. 	 
\end{proof}
 We have recalled all the tools to study the stochastic-periodic homogenization of non-stationary Navier-Stokes equations (\ref{tc11}) in heterogeneous media.

%%%%%%%%%%%%%%%%%%%%%%%%%%%%%%%%%%%%%%%%%%%%%%%%%%%%%%%%%%%%%%%%%%%%%%%%%%%%%%%%%%%%%
%%%%%%%%%%%%%%%%%%%%    4eme PARTIE   %%%%%%%%%%%%%%%%%%%%%%%%%%%%%%%%%%%%%%%%%%%%%%%

\section{Homogenized Navier-Stokes system}\label{sect4}
Let $\mathcal{T}$ be an $N$-dimensional dynamical system acting on the probability space $(\Omega, \mathscr{M}, \mu)$, $Q$ be a smooth bounded open subset of $\mathbb{R}^{N}_{x}$, where $N$ is a given positive integer, and let $T>0$ be a real number.
We are interested in the problem of the stochastic-periodic homogenization in heterogeneous media of non-stationary Navier-Stokes equations (\ref{tc11}) : 
\begin{equation*}
\begin{array}{rcll}
\rho \dfrac{\partial}{\partial t}\mathbf{u} + (\mathbf{u},\nabla)\mathbf{u} + \mathcal{E}(\mathbf{u}) + \nabla p & = & \mathbf{f} & \textup{in} \; [0,T]\times Q\times\Omega  \\
\textup{div} \mathbf{u} & = & 0 & \textup{in} \; [0,T]\times Q\times\Omega,
\end{array}
\end{equation*}
 To do this, we start firsly by constructing the \hbox{\LARGE$\epsilon$}-model for this problem. We recall that, here, for all $(t,x) \in [0,T]\times Q$ and for almost all $\omega\in \Omega$, $\mathbf{u}(t, x, \omega) = (u_{1}, u_{2}, \cdots, u_{N})$ denotes the \textit{Eulerian velocity} of a fluid flow, $p(t, x, \omega)$ is the scalar value \textit{fluid pressures}, $\rho(x, \omega)$ is a value scalar function, $\mathbf{f}(t, x, \omega) = (f_{1}, f_{2}, \cdots, f_{N})$ is $\mathbb{R}^{N}$-value vector function  and $(\mathbf{u},\nabla)u_{k} = \sum_{i=1}^{N} u_{i}\frac{\partial}{\partial x_{i}}u_{k}, k=1,2,\cdots,N$. The \textit{elasticity} operator $\mathcal{E}$ is described in the next paragraph.

\subsection{The $\epsilon$-model for (\ref{tc11})}

 We consider a structure consisting of fissures and matrices periodically distributed in a domain $Q$ in $\mathbb{R}^{N}$ with period $\epsilon^{2} Y$, where $\epsilon>0$ and $Y\equiv [0,1]^{N}$ is the unit cube. Let $Y$ be given in complementary parts, $Y_{1}$ and $Y_{2}$ which represent the fissure and matrix, respectively. Denote by $\chi_{m}(y)$ the characteristic function of $Y_{m}$ for $m=1,2$, extended $Y$-periodically to all $\mathbb{R}^{N}$. Thus, $\chi_{1}(y) + \chi_{2}(y)=1$. We shall assume that the set $\{ y \in \mathbb{R}^{N} : \chi_{1}(y)=1 \}$ is connected and smooth. 
\par The domain $Q$ is thus divided into the two subdomains, $Q_{1}^{\epsilon}$ and $Q_{2}^{\epsilon}$ representing the \textit{fissure} and \textit{matrix}, respectively, and given 
\begin{equation}
Q_{m}^{\epsilon} = \left\{ x \in Q : \chi_{m}\left(\frac{x}{\epsilon^{2}}\right)=1 \right\}, \quad m=1,2.
\end{equation}
Let $\Gamma_{12}^{\epsilon} \equiv \partial Q^{\epsilon}_{1}\cap Q^{\epsilon}_{2}\cap Q$ be the part of the interface of $Q^{\epsilon}_{1}$ with $Q^{\epsilon}_{2}$ that is interior to $Q$, and let $\Gamma_{12} \equiv \partial Y_{1}\cap \partial Y_{2}\cap Y$ be the corresponding part in the cell $Y$. Likewise, let $\Gamma_{22} \equiv Y_{2}\cap \partial Y$ and denoted by $\Gamma_{22}^{\epsilon}$ its periodic extension which forms the interface with those parts of the matrix $Q^{\epsilon}_{2}$ which lie within neighboring $\epsilon^{2} Y$-cells. These are the \textit{local blocks} and we denote them by $Y_{2}^{\epsilon}$. Finally, we set $Q_{m}^{\epsilon} \times \Omega \equiv \{ (x,\omega) \in Q\times \Omega : x \in Q_{m}^{\epsilon} \}$. Likewise, we define the set $\Gamma_{12}^{\epsilon}\times \Omega$, $\Gamma_{12}\times \Omega$, $\Gamma_{22}\times \Omega$ and $\Gamma_{22}^{\epsilon}\times \Omega$. 
\par We construct a system consisting of a non-stationary incompressible Navier-Stockes system in $Q_{1}^{\epsilon} \times \Omega$ coupled across the interface $\Gamma_{12}^{\epsilon}\times \Omega$ to another non-stationary Navier-Stockes system in $Q_{2}^{\epsilon} \times \Omega$. The stucture of fissured medium produces very high frequency spatial variations of pressures in the matrix and fissures, and so leads to corresponding variations of velocity fields. 
\par In order to describe these, the fluid \textit{pressures} in $Q_{m}^{\epsilon} \times \Omega$ are denoted by $p_{m}^{\epsilon}(t, x, \omega)$ at the position $x \in Q_{m}^{\epsilon}$, $m=1,2$ at time $t$ and for $\mu$-a.e. $\omega\in \Omega$. We set $\textbf{u}_{m}(t, x, \omega) = (u_{m1}, u_{m2}, \cdots, u_{mN})$ be the \textit{velocity field} at $x \in Q_{m}^{\epsilon}$, $m=1,2$ at time $t$ and for $\mu$-a.e. $\omega\in \Omega$, and $\varepsilon_{kl}(\textbf{u}) \equiv \frac{1}{2} \left(\frac{\partial}{\partial x_{k}}u_{l} + \frac{\partial}{\partial x_{l}}u_{k} \right)$ be the (linearized) \textit{strain} tensor, which measures the \textit{local deformation of velocity}. The \textit{stress} $\sigma(\textbf{u})$ is a necessarily symmetric tensor that represents the internal forces on surface elements resulting from such deformations, and we assume that the material is governed by the \textit{generalized Hooke's law} 
\begin{equation}
\sigma_{ij}^{m}(\textbf{u}_{m}) = \sum_{k,l=1}^{N} a_{ijkl}^{m}\varepsilon_{kl}(\textbf{u}_{m}), \quad m=1,2.
\end{equation}
The positive definite symmetric \textit{elasticity} tensor $a_{ijkl}$ provides a model for general anisotropic materials. We assume that for $\mu$-a.e. $\omega \in \Omega$ and $(x, y) \in Q\times Y$, $a_{ijkl}^{m}(x,\omega, y)$, $m=1,2$ are bounded continuous such that 
\begin{equation}
c_{1} \sum_{i,j,k,l=1}^{N} \eta_{ij} \eta_{kl} \leq \sum_{i,j,k,l=1}^{N} a_{ijkl}^{m}(x, \omega, y)\eta_{ij} \eta_{kl} \leq  c_{2} \sum_{i,j,k,l=1}^{N} \eta_{ij} \eta_{kl}, 
\end{equation}
where $(\eta_{ij})$ is an arbitrary symmetric matrix and $c_{1}, c_{2} >0$. The boundary conditions will involve with the surface density of forces or \textit{traction} $\sum_{j=1}^{N} \sigma_{ij}\eta_{j}$ determined by the unit normal vector $\textbf{n} = (n_{1}, n_{2}, \cdots, n_{N})$ on any boundary or interface. The normal will be disrected \textit{out} of $Q_{2}^{\epsilon}$. The eleastic structure is described for $\mu$-a.e. $\omega\in\Omega$, by bilinear forms 
\begin{equation}
e^{\varepsilon}_{m}(\textbf{u},\textbf{v}) \equiv \sum_{i,j,k,l=1}^{N} \int_{Q^{\epsilon}_{m}\times\Omega} a_{ijkl}^{m}\left(x, \mathcal{T}\left(\frac{x}{\epsilon}\right)\omega, \frac{x}{\epsilon^{2}} \right) \varepsilon_{kl}(\textbf{u}) \epsilon_{ij}(\textbf{v}) dxd\mu, 
\end{equation}
$m=1, 2$, on the space 
\begin{equation}
\begin{array}{c}
\mathbf{V} \equiv \left\{ \mathbf{v} \in \left[L^{2}(\Omega, H^{1}(Q)) \right]^{N} : \mathbf{v}(\cdot, \omega) = 0 \;\, \textup{on} \,\, \Gamma_{0} \,\; \textup{and} \, \textup{for} \, \mu-\textup{a.e.} \, \omega \in \Omega  \right\}  \\

\\
 \Gamma_{0} \subset \partial Q,
 \end{array}
\end{equation}
of \textit{admissible velocity fields of fluid}. 
\par The local description is obtained for $\mu$-a.e. $\omega\in\Omega$, by means of Green's theorem 
\begin{equation}
\begin{array}{rcl}
e^{\varepsilon}_{m}(\textbf{u}, \textbf{v}) & = & \int_{Q^{\epsilon}_{m}\times\Omega} \mathcal{E}^{\varepsilon}_{m}(\mathbf{u}(x,\omega))\cdot\mathbf{v}(x,\omega)dxd\mu \\
&   &  \\
 & & + (-1)^{m}\sum_{i,j=1}^{N} \int_{\Gamma^{\epsilon}_{12}\times\Omega} \sigma_{ij}^{m}\left(x, \mathcal{T}\left(\frac{x}{\epsilon} \right)\omega, \frac{x}{\epsilon^{2}} \right)\left(\mathbf{u}(s,\omega) \right)n_{j}v_{i}(s)\, dS,
\end{array}
\end{equation}
where the formal operator $\mathcal{E}^{\varepsilon}_{m}$ is given by 
\begin{equation}
\mathcal{E}^{\varepsilon}_{m}(\mathbf{u})_{i} = - \sum_{j=1}^{N} \sum_{k,l=1}^{N} \partial_{j}\, a^{m}_{ijkl}\left(x, \mathcal{T}\left(\frac{x}{\epsilon} \right)\omega, \frac{x}{\epsilon^{2}} \right) \varepsilon_{kl}(\mathbf{u}), \quad 1 \leq i \leq N, \quad m=1,2,
\end{equation}
whenever $\mathbf{u}, \mathbf{v} \in \mathbf{V}$ and $\mathcal{E}^{\varepsilon}_{m}(\mathbf{u}) \in [L^{2}(Q^{\epsilon}_{m} \times \Omega)]^{N}$.
\par Now, we are in the position to present the \hbox{\LARGE$\epsilon$}-model originated from the high frequency spatial variations of pressures and velocity fields in the matrix and fissures. The \hbox{\LARGE$\epsilon$}-model for diffusion on a \textit{Navier-Stokes fissured medium} is as follows : 
\begin{eqnarray}\label{tc8}
\rho_{1}\dfrac{\partial}{\partial t}\mathbf{u}_{1}^{\epsilon} + (\mathbf{u}_{1}^{\epsilon}, \nabla)\mathbf{u}_{1}^{\epsilon} + \mathcal{E}^{\varepsilon}_{1}(\mathbf{u}_{1}^{\epsilon}) + \nabla p_{1}^{\epsilon} = \mathbf{f}_{1}, & \quad \textup{in} \;\, Q_{1}^{\epsilon}\times\Omega, \label{tc1} \\
\textup{div} \mathbf{u}_{1}^{\epsilon}(\cdot,\omega) = 0, &  \textup{for \; a.e.} \;\, \omega\in \Omega \label{tc2}   \\
\mathbf{u}_{1}^{\epsilon}(\cdot,\omega) = \mathbf{u}_{2}^{\epsilon}(\cdot,\omega), \quad p_{1}^{\epsilon}(\cdot,\omega) = p_{2}^{\epsilon}(\cdot,\omega), &  \textup{for \; a.e.} \;\, \omega\in \Omega  \label{tc3} \\
\sum_{j=1}^{N} \sigma_{ij}(\mathbf{u}_{1}^{\epsilon}(\cdot,\omega))n_{j} = \sum_{j=1}^{N} \sigma_{ij}(\mathbf{u}_{2}^{\epsilon}(\cdot,\omega))n_{j},& \textup{for\,a.e.}\,\omega\in\Omega,\,1\leq i\leq N  \label{tc4}  \\
(\mathbf{u}_{1}^{\epsilon}\otimes \mathbf{u}_{1}^{\epsilon})\cdot\mathbf{n} = (\mathbf{u}_{2}^{\epsilon}\otimes \mathbf{u}_{2}^{\epsilon})\cdot\mathbf{n}, &  \textup{on} \;\, \Gamma^{\epsilon}_{12}\times\Omega   \label{tc5} \\
\rho_{2}\dfrac{\partial}{\partial t}\mathbf{u}_{2}^{\epsilon} + (\mathbf{u}_{2}^{\epsilon}, \nabla)\mathbf{u}_{2}^{\epsilon} + \mathcal{E}^{\varepsilon}_{2}(\mathbf{u}_{2}^{\epsilon}) + \nabla p_{2}^{\epsilon} = \mathbf{f}_{2}, &  \textup{in} \;\, Q_{2}^{\epsilon}\times\Omega,  \label{tc6} \\
\textup{div} \mathbf{u}_{2}^{\epsilon}(\cdot,\omega) = 0, &  \textup{for \; a.e.} \;\, \omega\in \Omega  \label{tc7}   
\end{eqnarray}
We make the following remarks.
\begin{rem}
	$R_{1})$ The \textit{global pressure} on $Q$ is given as 
	\begin{equation}
	p^{\epsilon}(x,\omega) = \chi_{1}\left(\frac{x}{\epsilon^{2}}\right)p_{1}^{\epsilon}(x,\omega) + \chi_{2}\left(\frac{x}{\epsilon^{2}}\right)p_{2}^{\epsilon}(x,\omega),
	\end{equation}
	and the \textit{global stress} is given as
	\begin{equation}
	\sigma^{\epsilon}_{ij}(\mathbf{u}(x,\omega)) = \chi_{1}	\sigma^{1}_{ij}(\mathbf{u}_{1}(x,\omega)) + \chi_{2}\sigma^{2}_{ij}(\mathbf{u}_{2}(x,\omega)),
	\end{equation}
	for almost all $\omega \in \Omega$.  \\
	
	$R_{2})$  We need to supplement (\ref{tc8}) with \textit{boundary conditions} on $\partial Q$. Only those prescribed for $p_{1}^{\epsilon}$ and $\mathbf{u}^{\epsilon}$ will survive the limit process as $\epsilon\to 0$.  \\
	
	$R_{3})$ The equations (\ref{tc4})-(\ref{tc5}) are just of the continuity of velocity fields, pressure, stress and inertia.  \\
	
	$R_{4})$ The components of (\ref{tc6}) are given by 
	\begin{equation}
	\rho_{2}\dfrac{\partial}{\partial t} u^{\epsilon}_{2i} + \sum_{j=1}^{N} u^{\epsilon}_{2j} \dfrac{\partial}{\partial x_{j}} u^{\epsilon}_{2j} -  \sum_{k,l=1}^{N} \dfrac{\partial}{\partial x_{i}} a_{ijkl}\left(x, \mathcal{T}\left(\frac{x}{\epsilon}\right)\omega, \frac{x}{\epsilon^{2}} \right)\varepsilon_{kl}(\mathbf{u}_{2}^{\epsilon}) + \dfrac{\partial p_{2}^{\epsilon}}{\partial x_{i}} = f_{i},
	\end{equation}
	for almost all $\omega \in \Omega$, $i=1,2, \cdots, N$, and (\ref{tc1}) is similar.  \\
	
	$R_{5})$ We can allow the \textit{quasi-static} cases $\rho_{m}=0$ which are examples of \textit{degenerate} evolution equations. Also, we could modify the model to include a scaling by any positive power of $\epsilon$ for $\rho_{m}$, and then these would be lost in the limit. In this discussion, we permit only an elliptic condition on $\rho_{m}$ :
	\begin{equation}
	0 < c_{3} \leq \rho_{m}(x, \omega), \quad \textup{for \; a.e.} \; \omega \in \Omega
	\;\, \textup{and} \;\, x \in Q^{\epsilon}_{m},
	\end{equation}
	for some constant $c_{3}$.
\end{rem}
The next section is devoted to the existence of solutions of the variational problem for the \hbox{\LARGE$\epsilon$}-model (\ref{tc1})-(\ref{tc7}) of non-stationary Navier-Stokes equations. 

\subsection{Existence of solutions to the weak formulation of (\ref{tc1})-(\ref{tc7})}

Let us begin to find the \textit{variational (weak) formulation} for the \hbox{\LARGE$\epsilon$}-model (\ref{tc1})-(\ref{tc7}) of non-stationary Navier-Stokes equations. Let 
\begin{equation}
\mathbf{u}^{\epsilon}_{m} \in L^{2}\left( [0,T] ; L^{2}(\Omega, H^{1}(Q^{\epsilon}_{m}))^{N} \right) \cap \mathcal{C}_{W}\left( [0,T] ; L^{2}(Q^{\epsilon}_{m}\times\Omega)^{N} \right), \quad m=1, 2
\end{equation} 
with $p^{\epsilon}_{1} = p^{\epsilon}_{2}$ and $\mathbf{u}^{\epsilon}_{1} = \mathbf{u}^{\epsilon}_{2}$ on $\Gamma^{\epsilon}_{12}\times\Omega$ such that 
\begin{equation}\label{tc9}
\begin{array}{l}
\sum_{m=1}^{2} \left[ \int_{Q^{\epsilon}_{m}\times\Omega}\rho_{m}\dfrac{\partial}{\partial t}\mathbf{u}^{\epsilon}_{m}\cdot\mathbf{v}_{m} + (\mathbf{u}^{\epsilon}_{m}, \nabla)\mathbf{u}^{\epsilon}_{m}\cdot\mathbf{v}_{m} \, dxd\mu + e^{\varepsilon}_{m}(\mathbf{u}_{m}, \mathbf{v}_{m}) \right] \\

\\
 = \sum_{m=1}^{2} \int_{Q^{\epsilon}_{m}\times\Omega} \mathbf{f}_{m}\cdot\mathbf{v}_{m}\,dxd\mu
\end{array}
\end{equation}
and
\begin{equation}\label{tc10}
\sum_{m=1}^{2} \left[ \int_{Q^{\epsilon}_{m}\times\Omega} (\textup{div}\mathbf{u}^{\epsilon}_{m})\varphi_{m}\,dxd\mu \right] = 0
\end{equation} 
for all divergence free vector fields $\mathbf{v}_{m} \in \mathcal{C}^{1}\left( [0,T] ; L^{2}(\Omega, H^{1}(Q^{\epsilon}_{m}))^{N} \right)$, $m=1,2$ with $\mathbf{v}_{1} = \mathbf{v}_{2}$ on $\Gamma^{\epsilon}_{12}\times\Omega$ and all $\varphi_{m} \in \mathcal{C}^{1}\left( [0,T] ; L^{2}(\Omega, H^{1}(Q^{\epsilon}_{m}))^{N} \right)$, $m=1,2$ with 
$\varphi_{1} = \varphi_{2}$ on $\Gamma^{\epsilon}_{12}\times\Omega$. \\
The space $\mathcal{C}_{W}\left( [0,T] ; L^{2}(Q^{\epsilon}_{m}\times\Omega)^{N} \right)$ is a subspace of $L^{\infty}\left( [0,T] ; L^{2}(Q^{\epsilon}_{m}\times\Omega)^{N} \right)$ consisting of functions which are weakly continuous : 
\begin{equation*}
t \to \int_{Q^{\epsilon}_{m}\times\Omega} u(t, x, \omega)\cdot h(x, \omega)dxd\mu
\end{equation*}
 is a continuous function, for all $h \in L^{2}(Q^{\epsilon}_{m}\times\Omega)^{N}$. 
\par In the following discussion, the function space $\mathbf{V}_{\textup{div}}$ is defined by 
\begin{equation}
\mathbf{V}_{\textup{div}} = \{ \mathbf{v} \in \mathbf{V} : \textup{div}\,\mathbf{v} = 0 \}.
\end{equation}
We have the following proposition which give the existence of solutions for system (\ref{tc9})-(\ref{tc10}) and whose proof, reverting to the deterministic case, can be found in \cite{pak2}.
\begin{prop}
	For each $\epsilon>0$ and $T>0$, there exists for $\mu$-a.e. $\omega\in \Omega$, a solution 
	\begin{equation}
	\mathbf{u}^{\epsilon}(\cdot, \omega) \in L^{2}([0,T] ; \mathbf{V}_{\textup{div}}) \cap \mathcal{C}_{W}([0,T] ; L^{2}(Q)^{N})
	\end{equation}
	of the variational formulation (\ref{tc9})-(\ref{tc10}) for every given $\mathbf{f} = \mathbf{f}_{0} + \textup{div}\mathbf{F}$ with 
	\begin{equation}
	\mathbf{f}_{0}(\cdot, \omega) \in L^{1}([0,T] ; L^{2}(Q)^{N}), \quad \mathbf{F}(\cdot, \omega) \in L^{2}([0,T] ; L^{2}(Q)^{N}),
	\end{equation}
	and divergence free initial vector field $\mathbf{u}^{\epsilon}(0, \omega) = \mathbf{u}^{\epsilon}_{0}(\omega) \in \mathbf{V}_{\textup{div}}$, for $\mu$-a.e. $\omega \in \Omega$.
\end{prop}

We investigate now  the boundness of the solutions $\mathbf{u}^{\epsilon}_{m}$, $m=1,2$ in (\ref{tc9})-(\ref{tc10}). \\
So, we substitue $\mathbf{v}_{1} = \mathbf{u}^{\epsilon}_{1}$, $\mathbf{v}_{2} = \mathbf{u}^{\epsilon}_{2}$ into (\ref{tc9}) to get (see \cite{pak2}), for all $\epsilon>0$ and for $dt\times d\mu$-a.e. $(t, \omega) \in [0, T]\times\Omega$,
\begin{equation}
\begin{array}{ll}
\sum_{m=1}^{2} \left[ \frac{1}{2}\left\|\sqrt{\rho_{m}}\mathbf{u}^{\epsilon}_{m}(t,\omega) \right\|^{2}_{L^{2}(Q_{m}^{\epsilon})^{N}} + \int_{0}^{t} e^{\varepsilon}_{m}(\mathbf{u}_{m}(\tau,\omega), \mathbf{u}_{m}(\tau,\omega)) d\tau \right] &    \\
    & \\
\leq \sum_{m=1}^{2} \left[ \frac{1}{2}\left\|\sqrt{\rho_{m}}\mathbf{u}^{\epsilon}_{m}(0,\omega) \right\|^{2}_{L^{2}(Q_{m}^{\epsilon})^{N}} + \int_{0}^{t}\int_{Q_{m}^{\epsilon}} \mathbf{f}_{m}(\tau,x,\omega)\cdot\mathbf{u}_{m}(\tau,x,\omega) dxd\tau \right] &    \\
    &  \\
\leq \sum_{m=1}^{2} \bigg[ \frac{1}{2}\left\|\sqrt{\rho_{m}}\mathbf{u}^{\epsilon}_{m}(0,\omega) \right\|^{2}_{L^{2}(Q_{m}^{\epsilon})^{N}}  \\  
\qquad \qquad\qquad\qquad + \int_{0}^{t} \left\|\mathbf{f}_{m}(\tau,\omega)\right\|^{2}_{L^{2}(Q_{m}^{\epsilon})^{N}}\cdot\left\|\mathbf{u}_{m}(\tau,\omega)\right\|^{2}_{L^{2}(Q_{m}^{\epsilon})^{N}} d\tau \bigg].  & 
\end{array}
\end{equation}
Thus, as in \cite[page 6]{pak2}, applying Korn's inequality (see \cite{olei}), we see that, for $\mu$-a.e. $\omega\in\Omega$, 
\begin{equation}\label{rem6}
\left\| \mathbf{u}^{\epsilon}_{1}(\cdot,\omega)\right\|_{L^{2}([0,T] ; L^{2}(Q_{1}^{\epsilon})^{N}}) \quad \textup{and} \quad \left\| \mathbf{u}^{\epsilon}_{2}(\cdot,\omega)\right\|_{L^{2}([0,T] ; L^{2}(Q_{2}^{\epsilon})^{N}})
\end{equation}
are bounded. These facts, in turn, imply that for $\mu$-a.e. $\omega\in\Omega$,
\begin{equation}\label{rem7}
\begin{array}{ll}
\left\| e_{kl}(\mathbf{u}^{\epsilon}_{1}(\cdot,\omega))\right\|_{L^{2}([0,T] ; L^{2}(Q_{1}^{\epsilon})^{N^{2}})} \quad \textup{,} \quad \left\| e_{kl}(\mathbf{u}^{\epsilon}_{2}(\cdot,\omega))\right\|_{L^{2}([0,T] ;  L^{2}(Q_{2}^{\epsilon})^{N^{2}})} & \\
&  \\
\left\| \mathbf{u}^{\epsilon}_{1}(\cdot,\omega)\right\|_{L^{\infty}([0,T] ; L^{2}(Q_{1}^{\epsilon})^{N})} \quad \textup{and} \quad \left\| \mathbf{u}^{\epsilon}_{2}(\cdot,\omega)\right\|_{L^{\infty}([0,T] ; L^{2}(Q_{2}^{\epsilon})^{N})} &
\end{array}
\end{equation}
are bounded. 
\par After construction of weak formulation of \hbox{\LARGE$\epsilon$}-model of non-stationary Navier-Stokes equations (\ref{tc11}), we shall study in the next section the stochastic-homogenization of this weak formulation.

\subsection{Main homogenization results for (\ref{tc9})-(\ref{tc10})}

In order to homogenize the \textit{weak formulation} (\ref{tc9})-(\ref{tc10}), we can introduce the stochastic two-scale limit for the strain operator using the compactness theorem given in proposition \ref{rem12}. 
\par We present stochastic two-scale strain behavior of bounded vector fields for which its strain is bounded. We denote the strain operator by $\mathbf{e}(\mathbf{u})$, that is, the $i-j$ component $\mathbf{e}(\mathbf{u})_{ij}$ of $\mathbf{e}(\mathbf{u})$ is 
\begin{equation}
\mathbf{e}(\mathbf{u})_{ij} = \varepsilon_{ij}(\mathbf{u}) = \frac{1}{2}\left( \dfrac{\partial \mathbf{u}_{i}}{\partial x_{j}} + \dfrac{\partial \mathbf{u}_{j}}{\partial x_{i}} \right),
\end{equation} 
and also 
\begin{equation}
\mathbf{e}(\mathbf{u})_{y} = \varepsilon_{ij}(\mathbf{u}) =  \frac{1}{2}\left( \dfrac{\partial \mathbf{u}_{i}}{\partial y_{j}} + \dfrac{\partial \mathbf{u}_{j}}{\partial y_{i}} \right).
\end{equation}
Similarly, we define the stochastic analog $\mathbf{e}(\mathbf{u})_{\omega}$ of $\mathbf{e}(\mathbf{u})_{y}$ by 
\begin{equation}
\mathbf{e}(\mathbf{u})_{\omega} = \frac{1}{2}\left( \overline{D}^{\omega}_{j}\mathbf{u}_{i} + \overline{D}^{\omega}_{i}\mathbf{u}_{j} \right).
\end{equation} 
We have the following lemma which gives us the first convergence result (componentwise) for the strain operator.
\begin{lem}\label{rem8}
	Assume that a sequence of vector fields $\{\mathbf{u}^{\epsilon}\}_{\epsilon\in E}$ is bounded in $L^{2}(Q\times\Omega)^{N}$ and the sequence of matrices $\{\mathbf{e}(\mathbf{u}^{\epsilon}) \}$ is bounded in $L^{2}(Q\times\Omega)^{N^{2}}$, respectively. Then, there is a subsequence  $E'$ from $E$ such that, as $E' \ni \epsilon \rightarrow 0$ we have :
	\begin{itemize}
		\item[i)] $\mathbf{u}^{\epsilon} \rightarrow \mathbf{u}$ stoch. in $L^{2}(Q\times\Omega)^{N}$-weak 2s,
		\item[ii)]  $\mathbf{e}(\mathbf{u}^{\epsilon}) \rightarrow \mathbf{e}(\mathbf{u}) + \mathbf{e}(\mathbf{u}_{1})_{\omega} + \mathbf{e}(\mathbf{u}_{2})_{y}$ stoch. in $L^{2}(Q\times\Omega)^{N}$-weak 2s,
	\end{itemize}
	for some $\mathbf{u}(x,\omega) \in H^{1}(Q ; L^{2}_{nv}(\Omega))^{N}$, $\mathbf{u}_{1} \in L^{2}(Q ; H^{1}_{\#}(\Omega))^{N}$ and $\mathbf{u}_{2} \in L^{2}(Q\times\Omega ; H^{1}_{\#}(Y))^{N}$.
\end{lem}
\begin{proof}
	It can be proved by Korn's inequality in (\ref{rem6})-(\ref{rem7}) and the compactness theorem in proposition \ref{rem12}.	
\end{proof}

Now, we will consider the limit equations for Navier-Stokes system (\ref{tc9})-(\ref{tc10}). Let us denote the scaled characteristic functions by 
\begin{equation}
\chi_{m}^{\epsilon} \equiv \chi_{m}\left(\frac{x}{\epsilon^{2}}\right), \quad m=1, 2.
\end{equation}

Firstly, proposition \ref{rem12}, lemma \ref{rem8} and a priori estimates yield the existence of subsequences (it is also denoted by $\mathbf{u}^{\epsilon}$) which stochastically two-scale converge to some 
\begin{eqnarray}
\mathbf{u} \in L^{\infty}\left([0,T] ; H^{1}(Q ; L^{2}_{nv}(\Omega))^{N} \right)  \\
\mathbf{u}_{1} \in L^{\infty}\left([0,T] ; L^{2}(Q ; H^{1}_{\#}(\Omega))^{N} \right) \\
\mathbf{u}_{2} \in L^{\infty}\left([0,T] ; L^{2}(Q\times\Omega ; H^{1}_{\#}(Y))^{N} \right)
\end{eqnarray}
as follows
\begin{eqnarray}
\mathbf{u}^{\epsilon}(t,x,\omega) \rightarrow \mathbf{u}(t,x,\omega) \; \textup{stoch.\,in} \, L^{2}(Q\times\Omega)^{N}-\textup{weak\,2s} \\
\mathbf{e}(\mathbf{u}^{\epsilon}) \rightarrow \mathbf{e}_{kl}(\mathbf{u}) + \mathbf{e}_{kl}^{\omega}(\mathbf{u}_{1}) + \mathbf{e}_{kl}^{y}(\mathbf{u}_{2}) \; \textup{stoch.\,in} \, L^{2}(Q\times\Omega)^{N}-\textup{weak\,2s}
\end{eqnarray}
where 
\begin{eqnarray}
\mathbf{e}_{kl}^{y}(\mathbf{u}_{2}) = \dfrac{1}{2}\left(\dfrac{\partial \mathbf{u}_{2l}}{\partial y_{k}} + \dfrac{\partial \mathbf{u}_{2k}}{\partial y_{l}} \right) \\
\mathbf{e}_{kl}^{\omega}(\mathbf{u}_{1}) = \dfrac{1}{2}\left(\overline{D}^{\omega}_{k}\mathbf{u}_{1l} + \overline{D}^{\omega}_{l}\mathbf{u}_{1k} \right).
\end{eqnarray}

Secondly, by virtue of Korn's inequality (see \cite[page 9]{pak2}), we get that, for $\mu$-a.e. $\omega\in \Omega$ 
\begin{equation}
\left\|\mathbf{u}^{\epsilon}_{1}(\cdot,\omega) \right\|_{L^{\infty}([0,T] ; H^{1}(Q^{\epsilon}_{1})^{N})}, \quad \left\|\mathbf{u}^{\epsilon}_{2}(\cdot,\omega) \right\|_{L^{\infty}([0,T] ; H^{1}(Q^{\epsilon}_{2})^{N})}
\end{equation}
are bounded, so then Rellich Theorem yields $\mathbf{u}^{\epsilon}(t,\omega)$ \textit{converges strongly} to $\mathbf{u}(t,\omega)$ in $L^{2}(Q)^{N}$, for each $t \in [0,T]$ and for $\mu$-a.e. $\omega\in \Omega$. 
\par Thirdly, the divergence free condition of velocity fields $\mathbf{u}^{\epsilon}_{m}$ in (\ref{tc2}) and (\ref{tc7}) yields, using the proposition \ref{rem5},
\begin{equation}\label{rem9}
\textup{div}\mathbf{u}(x,\omega) + \textup{div}_{\omega}\mathbf{u}_{1}(x,\omega) + \textup{div}_{y}\mathbf{u}_{2}(x,\omega,y) = 0, \quad \textup{on}\;Q\times\Omega.
\end{equation} 
We take the weak limit on the condition (\ref{tc10}) to get $\textup{div}\mathbf{u}=0$. Therefore, from (\ref{rem9}), we have 
\begin{equation}
\textup{div}_{\omega}\mathbf{u}_{1}(x,\omega) + \textup{div}_{y}\mathbf{u}_{2}(x,\omega,y) = 0, \quad \textup{on}\;Q\times\Omega.
\end{equation} 
We deduce that $\textup{div}_{\omega}\mathbf{u}_{1}=0$ and $\textup{div}_{y}\mathbf{u}_{2}=0$.\\

Now, the \textit{variational} \hbox{\LARGE$\epsilon$}-equation (\ref{tc9}) can be displayed as, 
\begin{equation}\label{rem10}
\begin{array}{l}
\int_{Q\times\Omega} \chi_{1}^{\epsilon}(x) \bigg\{ \rho_{1}(x,\omega)\dfrac{\partial}{\partial t}\mathbf{u}^{\epsilon}_{1}(x,\omega)\cdot\mathbf{v}_{1}^{\epsilon}(x,\omega) \\
\qquad \qquad \quad + (\mathbf{u}^{\epsilon}_{1}(x,\omega), \nabla)\mathbf{u}^{\epsilon}_{1}(x,\omega)\cdot\mathbf{v}_{1}^{\epsilon}(x,\omega) \bigg\} dxd\mu \\
 + \int_{Q\times\Omega} \chi_{2}^{\epsilon}(x) \bigg\{ \rho_{2}(x,\omega)\dfrac{\partial}{\partial t}\mathbf{u}^{\epsilon}_{2}(x,\omega)\cdot\mathbf{v}_{2}^{\epsilon}(x,\omega) \\
\qquad \qquad \quad + (\mathbf{u}^{\epsilon}_{2}(x,\omega), \nabla)\mathbf{u}^{\epsilon}_{2}(x,\omega)\cdot\mathbf{v}_{2}^{\epsilon}(x,\omega) \bigg\} dxd\mu  \\
 +\, e^{\varepsilon}_{1}(\mathbf{u}_{1}, \mathbf{v}_{1}^{\epsilon}) + e^{\varepsilon}_{2}(\mathbf{u}_{2}, \mathbf{v}_{2}^{\epsilon}) \\
 
 \\
=  \int_{Q\times\Omega} \chi_{1}^{\epsilon}(x) \mathbf{f}_{1}(x,\omega)\cdot\mathbf{v}_{1}^{\epsilon}(x,\omega) \,dxd\mu \\
 
 \\
\qquad \qquad \quad + \int_{Q\times\Omega} \chi_{2}^{\epsilon}(x) \mathbf{f}_{2}(x,\omega)\cdot\mathbf{v}_{2}^{\epsilon}(x,\omega) \,dxd\mu
\end{array}
\end{equation}
for all divergence free vector fields \textit{admissible} $\mathbf{v}_{m}^{\epsilon} = \mathbf{v}_{m}\left(t,x,\mathcal{T}\left(\frac{x}{\epsilon}\right)\omega, \frac{x}{\epsilon^{2}} \right)  \in \mathcal{C}^{1}\left( [0,T] ; L^{2}(\Omega, H^{1}(Q^{\epsilon}_{m}))^{N} \right)$, $m=1,2$ with
\begin{equation}
\mathbf{v}_{1} = \mathbf{v}_{2} \, \textup{on} \, \Gamma^{\epsilon}_{12}\times\Omega \, \textup{and} \; \mathbf{v}_{1}(T) = \mathbf{v}_{2}(T)=0. 
\end{equation}
Using integration by parts on the variable $t$, we rewrite (\ref{rem10}) as, 
\begin{equation}\label{rem11}
\begin{array}{l}
\int_{0}^{T}\int_{Q\times\Omega} \chi_{1}^{\epsilon}(x)  \rho_{1}(x,\omega)\mathbf{u}_{1}^{\epsilon}(t,x,\omega)\cdot\dfrac{\partial}{\partial t}\mathbf{v}^{\epsilon}_{1}(t,x,\omega)  +  \int_{0}^{T} e^{\varepsilon}_{1}(\mathbf{u}_{1}(t), \mathbf{v}_{1}^{\epsilon}(t))\,dt  \\
\qquad + \int_{0}^{T}\int_{Q\times\Omega} \chi_{1}^{\epsilon}(x) \{ \mathbf{u}^{\epsilon}_{1}(t,x,\omega)\otimes\mathbf{u}^{\epsilon}_{1}(t,x,\omega)\}\cdot\nabla \mathbf{v}_{1}^{\epsilon}(t,x,\omega)  dxd\mu dt \\
\qquad +\int_{0}^{T}\int_{Q\times\Omega} \chi_{2}^{\epsilon}(x)  \rho_{2}(x,\omega)\mathbf{u}_{2}^{\epsilon}(t,x,\omega)\cdot\dfrac{\partial}{\partial t}\mathbf{v}^{\epsilon}_{2}(t,x,\omega)  +  \int_{0}^{T} e^{\varepsilon}_{2}(\mathbf{u}_{2}(t), \mathbf{v}_{2}^{\epsilon}(t))\,dt  \\
\qquad + \int_{0}^{T}\int_{Q\times\Omega} \chi_{2}^{\epsilon}(x) \{ \mathbf{u}^{\epsilon}_{2}(t,x,\omega)\otimes\mathbf{u}^{\epsilon}_{2}(t,x,\omega)\}\cdot\nabla \mathbf{v}_{2}^{\epsilon}(t,x,\omega)  dxd\mu dt \\

\\
= \int_{Q_{1}^{\epsilon}\times\Omega} \rho_{1}(x,\omega)\mathbf{u}_{1}^{\epsilon}(0,x,\omega)\cdot\mathbf{v}_{1}^{\epsilon}(0,x,\omega) dxd\mu \\
\qquad + \int_{Q_{2}^{\epsilon}\times\Omega} \rho_{2}(x,\omega)\mathbf{u}_{2}^{\epsilon}(0,x,\omega)\cdot\mathbf{v}_{2}^{\epsilon}(0,x,\omega) dxd\mu \\
\qquad +\int_{0}^{T}\int_{Q\times\Omega} \chi_{1}^{\epsilon}(x) \mathbf{f}_{1}(t,x,\omega)\cdot\mathbf{v}_{1}^{\epsilon}(t,x,\omega) \,dxd\mu dt \\
\qquad + \int_{0}^{T}\int_{Q\times\Omega} \chi_{2}^{\epsilon}(x) \mathbf{f}_{2}(t,x,\omega)\cdot\mathbf{v}_{2}^{\epsilon}(t,x,\omega) \,dxd\mu dt.
\end{array}
\end{equation}
Thus, applying any divergence free test field \textit{admissible} \\ $\mathbf{v} \in L^{2}(\Omega, H^{1}(Q^{\epsilon}_{m}))^{N}$ in the place of $\mathbf{v}_{1}$ and $\mathbf{v}_{2}$, we have, using the compactness theorem \ref{rem2}, the weak limit of \hbox{\LARGE$\epsilon$}-equation (\ref{rem11}) : 
\begin{equation}\label{rem13}
\begin{array}{l}
\int_{0}^{T}\int_{Q\times\Omega} \bigg\{  \rho(x,\omega) \dfrac{\partial}{\partial t}\mathbf{u}(t,x,\omega)\cdot \mathbf{v}(t,x,\omega) \\
\qquad\qquad \quad + \{ \mathbf{u}(t,x,\omega)\otimes\mathbf{u}(t,x,\omega)\}\cdot\nabla \mathbf{v}(t,x,\omega) \bigg\}dxd\mu dt  \\
\qquad +  \int_{0}^{T} e(\mathbf{u}(t) + \mathbf{u}_{1}(t) + \mathbf{u}_{2}(t), \mathbf{v}(t))\,dt \\

\\ 
= \int_{0}^{T}\int_{Q\times\Omega} \mathbf{f}(t,x,\omega)\cdot\mathbf{v}(t,x,\omega) \,dxd\mu dt
\end{array}
\end{equation}
with the \textit{effective coefficients} given by
\begin{equation}
\begin{array}{l}
\rho(x,\omega) = |Y_{1}|\rho_{1}(x,\omega) + |Y_{2}|\rho_{2}(x,\omega), \\

\\
 \mathbf{f}(x,\omega) = |Y_{1}|\mathbf{f}_{1}(x,\omega) + |Y_{2}|\mathbf{f}_{2}(x,\omega)
\end{array}
\end{equation}
and the homogenized elasticity bilinear form is defined by, 
\begin{equation}
\begin{array}{l}
e(\mathbf{u} + \mathbf{u}_{1} + \mathbf{u}_{2}, \mathbf{v}(t)) 
 \equiv \sum_{i,j,k,l=1}^{N} \int_{Y}\int_{Q\times\Omega} a_{ijkl}(x,\omega,y) \bigg\{ \varepsilon_{kl}(\mathbf{u}) \\
 
 \\
 \qquad\qquad\qquad \qquad\qquad\qquad + \varepsilon_{kl}^{\omega}(\mathbf{u}_{1}) + \varepsilon_{kl}^{y}(\mathbf{u}_{2})  \bigg\}\varepsilon_{ij}(\mathbf{v})\,dxd\mu dy,
\end{array}
\end{equation}
where the corresponding effective elasticity tensor is 
\begin{equation}
a_{ijkl}(x,\omega,y) = \chi_{1}(y)a^{1}_{ijkl}(x,\omega,y) + \chi_{2}(y)a^{2}_{ijkl}(x,\omega,y).
\end{equation}

\begin{rem}
	Indeed, when we consider the particular dynamical system $\mathcal{T}(y)$ on $\Omega = \mathbb{T}^{N} \equiv \mathbb{R}^{N}/\mathbb{Z}^{N}$ (the $N$-dimensional torus) defined by $\mathcal{T}(y)\omega = y+\omega\;\textup{mod}\;\mathbb{Z}^{N}$, then one can view $\Omega$ as the unit cube in $\mathbb{R}^{N}$ with all the pairs of antipodal faces being identified. The Lebesgue measure on $\mathbb{R}^{N}$ induces the Haar measure on $\mathbb{T}^{N}$ which is invariant with respect to the action of $\mathcal{T}(y)$ on $\mathbb{T}^{N}$. Moreover, $\mathcal{T}(y)$ is ergodic and in this situation, any function on $\Omega$ may be regarded as a periodic function on $\mathbb{R}^{N}$ whose period in each coordinate is 1, so that in this case
	 \textit{Eulerian velocity} $\mathbf{u}$, \textit{fluid pressures} $p$ and the functions $\rho$, $\mathbf{f}$ 
	  may be viewed as the periodic "functions" with respect to the variable $\omega$. Therefore, the stochastic-deterministic problem (\ref{tc11}) of Navier-Stokes is equivalent to the deterministic reiterated problem of Navier-Stokes and whose the non-reiterated stationary case is treated in \cite{pak2}.
\end{rem}

%So, the strong formulation of (\ref{rem13}) is given by 
%\begin{equation}
%\begin{array}{rcll}
%\rho(x, \omega) \dfrac{\partial}{\partial t}\mathbf{u}(t, x, \omega) + (\mathbf{u},\nabla)\mathbf{u}(t, x, \omega) + \mathcal{E}(\mathbf{u}(t, x, \omega)+\mathbf{u}_{1}(t, x, \omega)+\mathbf{u}_{2}(t, x,\omega,y))  & = & \mathbf{f}(t, x, \omega). &   
%\end{array}
%\end{equation}

%\vspace{0.5cm}

\flushleft	\textbf{Acknowledgements.} The authors would like to thank the anonymous referee for his/her pertinent remarks, comments and suggestions.

%-----------------------------------------------------Bibliographie--------------------------------------------------------------------
%\clearpage\addcontentsline{toc}{chapter}{Bibliographie}

% ------------------------------------------------------------------------
\end{document}